\documentclass[12pt,a4paper]{amsart}
\usepackage{amssymb,latexsym,amsmath,doublespace}
\usepackage[all]{xy}
\usepackage{multicol}
\usepackage{color}

\newcommand{\cK}{\mathcal{K}}     %%% algebra of compact operators
\newcommand{\cZ}{\mathcal{Z}}    
\newcommand{\C}{\mathbb{C}}        %%% complex numbers
\newcommand{\R}{\mathbb{R}}        %%% real  numbers       
\newcommand{\N}{\mathbb{N}}        %%% complex numbers
\newcommand{\otmin}{\otimes_{\mathrm{min}}}
\newcommand{\otmax}{\otimes_{\mathrm{max}}}
\newcommand{\otmininfty}{{\textstyle{\bigotimes_{\mathrm{min}}^\infty} \,}}
\newcommand{\otmaxinfty}{{\textstyle{\bigotimes_{\mathrm{max}}^\infty} \,}}

\newcommand{\Cs}{$C^*$-algebra}
\newcommand{\sh}{{$^*$-ho\-mo\-mor\-phism}}
\newcommand{\Ker}{\mathrm{Ker}}
\newcommand{\ep}{\varepsilon}
\newcommand{\Cu}{\mathrm{Cu}}
\newcommand{\nprecsim}{{\operatorname{\hskip3pt \precsim\hskip-9pt |\hskip6pt}}}
\newcommand{\wDiv}{{\mathrm{w\mbox{-}Div}}}
\newcommand{\Div}{{\mathrm{Div}}}
\newcommand{\Dec}{{\mathrm{Dec}}}

\newtheorem{thm}{Theorem}[section]
\newtheorem{cor}[thm]{Corollary}
\newtheorem{lemma}[thm]{Lemma}
\newtheorem{prop}[thm]{Proposition}

\theoremstyle{definition}
\newtheorem{definition}[thm]{Definition}

\newtheorem{ex}[thm]{Example}
\newtheorem{rem}[thm]{Remark}
\newtheorem{rems}[thm]{Remarks}
\newtheorem{ques}[thm]{Question}

\numberwithin{equation}{section}
\begin{document}

\title{When central sequence $C^*$-algebras have characters}
\author{Eberhard Kirchberg and  Mikael R{\o}rdam}
\date{November 18, 2014}

\thanks{The second named author was supported  by the Danish National Research Foundation (DNRF) through the Centre for Symmetry and Deformation at University of Copenhagen, and The Danish Council for Independent Research, Natural Sciences.}
%
%HEREXXX
%19K99, 46L80, 54B20, 28C15,
%06F30, 06B35, 60B05, 06D, 54D}
%
%
%%C*-algebras,
%%
%

\begin{abstract} We investigate \Cs s whose central sequence algebra has no characters, and we raise the question if such \Cs s necessarily must absorb the Jiang-Su algebra (provided that they also are separable). We relate this question to a question of Dadarlat and Toms if the Jiang-Su algebra always embeds into the infinite tensor power of any unital \Cs{} without characters. We show that absence of characters of the central sequence algebra implies that the \Cs{} has the so-called strong Corona Factorization Property, and we use this result to exhibit simple nuclear separable unital \Cs s whose central sequence algebra does admit a character. We show how stronger divisibility properties on the central sequence algebra imply stronger regularity properties of the underlying \Cs.
\end{abstract}

\maketitle

%\tableofcontents

%%%%%%%%%%%%%%%%%%%%%%%%%%%%%%%%%%%%%%%%%%%%%%%%

\section{Introduction}
\label{sec:intro}

\noindent Dusa McDuff proved in 1970 that the (von Neumann) central sequence algebra of a von Neumann II$_1$-factor is either abelian or a type II$_1$-von Neumann algebra. The latter holds if and only if the given II$_1$-factor tensorially absorbs the hyperfinite II$_1$-factor (as von Neumann algebras). Such II$_1$-factors are now called McDuff factors. Analogously, if $A$ is a separable unital \Cs{} and if $D$ is a separable unital \Cs{} for which the so-called "half-flip" is approximately inner, then $A$ is isomorphic to $A \otimes D$ if $D$ embeds unitally into the \Cs{} central sequence algebra $A_\omega \cap A'$, where $A_\omega$ denotes the (norm) ultrapower \Cs{} with respect to a given ultrafilter $\omega$ (see for example \cite[Theorem 7.2.2]{Ror:Encyc}.) If, moreover, $D$ is strongly self-absorbing, then $A \cong A \otimes D$ if \emph{and only if} $D$ embeds unitally into $A_\omega \cap A'$. We follow the notation of \cite{Kir:abel} and write $F(A)$ for the central sequence \Cs{} $A_\omega \cap A'$ (suppressing the choice of ultrafilter $\omega$, cf.\ Remark \ref{CH}).

The McDuff dichotomy for type II$_1$-von Neumann factors does not immediately carry over to (simple, unital, stably finite) \Cs s. The central sequence algebra in the world of \Cs s is more subtle. It is rarely abelian (cf.\ the recent paper by H.\ Ando and the first named author, \cite{HiKir:abelian}, where it is shown that the central sequence algebra is non-abelian whenever the original \Cs{} is not of type I); and when it is non-abelian it may not contain a unital copy of any unital simple \Cs{} other than $\C$. 

It was shown in \cite{KirPhi:classI} that if $A$ is a simple, unital, separable, nuclear, purely infinite \Cs, then $F(A)$ is simple and purely infinite. In particular, $\mathcal{O}_\infty$ embeds into $F(A)$, so $A$ is isomorphic to $A \otimes \mathcal{O}_\infty$. 

Significant progress in our understanding of the central sequence algebra in the stably finite case has recently been obtained by Matui and Sato in \cite{MatuiSato:Z} and \cite{MatuiSato:UHF}. A result by Sato, improved in \cite{KirRor:Central-sequence}, provides an epimomorphism from $F(A)$ onto the von Neumann central sequence algebra of the weak closure of $A$ with respect to any tracial state on $A$. In the case where $A$ is nuclear (and with no finite dimensional quotients) and the trace is extreme, we thus get an epimorphism from $F(A)$ onto the central sequence algebra of the hyperfinite II$_1$-factor, which is a type II$_1$-von Neumann algebra. Matui and Sato introduced a comparability property of the central sequence algebra, which they call (SI), and which, in the case of finitely many extremal traces, fascilitates liftings from this II$_1$-von Neumann algebra to $F(A)$ itself. They use this to show that $\cZ$-stability is equivlaent to strict comparison (of positive elements) for simple, unital, separable, nuclear \Cs s with finitely many extremal traces. 

The first named author observed in \cite{Kir:abel} that if $A$ is a separable unital \Cs, and $D$ is another separabel unital \Cs{} which via a \sh{} maps unitally into $F(A)$, then the infinite maximal tensor power $\otmaxinfty D$ likewise maps unitally into $F(A)$. This leads to the dichotomy that $F(A)$ either has a character or admits no finite-dimensional representations; and in the latter case there is a unital \sh{} from  the infinite maximal tensor power  of some separable unital \Cs{} $D$ without characters into $F(A)$. Dadarlat and Toms proved in \cite{DadToms:Z} that the Jiang-Su algebra $\cZ$ embeds unitally into $\otmininfty D$ (or into $\otmaxinfty D$) if and only if the latter contains a subhomogeneous \Cs{} without characters as a unital sub-\Cs. 

It is therefore natural to ask if the Jiang-Su algebra $\cZ$ embeds unitally into $F(A)$ if and only if $F(A)$ has no characters, whenever $A$ is a unital separable \Cs. As mentioned above, the former is equivalent to the isomorphism $A \cong A \otimes \cZ$. We remark that our question is equivalent to the question of Dadarlat and Toms if the Jiang-Su algebra always embeds into $\otmininfty D$, when $D$ is a unital \Cs{} without characters. 
We show, using results from \cite{RobRor:divisible},  that if $A$ is a unital separable \Cs{} for which $F(A)$ has no characters, then $A$ has the so-called strong Corona Factorization Property, and we use this to give examples of unital separable nuclear simple \Cs s $A$ for which $F(A)$ does have a character. 

We investigate stronger divisibility properties of the central sequence algebra, and show how they lead to stronger comparison  and divisibility properties of the given \Cs. We conclude by giving a necessary and sufficient divisibility condition on the central sequence algebra that it admits a unital embedding of the Jiang-Su algebra.

%%%%%%%%%%%%%%%%%%%%%%%%%%%%%%%%%%%%%%%%%%%%%%%%

\section{Preliminaries and a question of Dadarlat--Toms}

\noindent We recall in this section some well-known results that have motivated this paper.  

For each unital \Cs{}  $D$ consider the minimal and maximal infinite tensor powers
$$
\otmininfty D = D \otmin D \otmin D \otmin\cdots, \qquad
\otmaxinfty D =D \otmax D \otmax D \otmax\cdots.
$$
If $D$ is nuclear or if there is no need to specify which tensor product is being used, then we may drop the subscripts "min" and "max". 

For each free ultrafilter $\omega$  on $\N$, let $D_\omega$ denote the quotient $\ell^\infty(D)/c_\omega(D)$, where $c_\omega(D)$ is the closed two-sided ideal in $\ell^\infty(D)$ consisting of all bounded sequences $\{d_n\}_{n=1}^\infty$ such that $\lim_{n \to \omega}\|d_n\| = 0$. We shall denote the \emph{central sequence algebra} $D_\omega \cap D'$ by $F_\omega(D)$, or just by $F(D)$ ignoring the choice of free ultrafilter $\omega$, cf.\ the remark below. 

\begin{rem} \label{CH}  It was shown by Ge and Hadwin in \cite{GeHadwin:CH} that the isomorphism class of $F_\omega(A)$ is independent of the choice of free ultrafilter $\omega$, when $A$ is a separable \Cs{} and assuming that the continuum hypothesis (CH) holds. Farah, Hart and Sherman proved in \cite{FaHaSh:model} that if (CH) fails and $A$ is unital and separable, then there are $2^{2^{\aleph_0}}$ pairwise non-isomorphic central sequence algebras $F_\omega(A)$ (for different choices of free ultrafilters $\omega$). 

In general, whether or not (CH) holds, if $A$ is any \Cs, $B$ is a (unital) separable \Cs{} and there is a (unital) injective \sh{} $B \to F_\omega(A)$ for some free ultrafilter $\omega$, then there is a (unital) injective \sh{} $B \to F_{\omega'}(A)$ for any other free ultrafilter $\omega'$. 

It follows that properties, such as the existence of a unital \sh{} $\cZ \to F_\omega(A)$ and absence of characters on $F_\omega(A)$, is independent of the choice of free ultrafilter $\omega$. (Use Proposition~\ref{prop:A(n,2)} and the oberservation above to see the latter.) 
\end{rem}

\noindent
We mention below a result by the first-named author and a corollary thereof. 

\begin{prop}[{\cite[Proposition 1.12]{Kir:abel}}] \label{prop:Abel}
Let $\omega$ be a free ultrafilter on $\N$, let $A$ be a unital separable \Cs, let $B$ be a unital separable sub-\Cs{} of $A_\omega$, and let $D$ be a unital separable sub-\Cs{} of $A_\omega \cap A'$. It follows that there exists a unital $^*$-monomorphism $D \to A_\omega \cap B'$.
\end{prop}

\begin{cor} \label{cor:Abel}
Whenever $A$ is a unital separable \Cs{} and $D$ is a unital separable sub-\Cs{} of $F(A)$, there is a unital \sh{} $\otmaxinfty \, D \to F(A)$.
\end{cor}

\begin{proof} Find inductively unital \sh s $\varphi_n \colon D \to F(A)$ with pairwise commuting images as follows. Let $\varphi_1$ be the inclusion mapping $D \to F(A)$. If $n \ge 2$, then use Proposition~\ref{prop:Abel} with $B = C^*(A,\varphi_1(D), \dots, \varphi_{n-1}(D))$ to find $\varphi_n \colon D \to F(A)$ with the desired properties. 
\end{proof}

\noindent Dadarlat and Toms proved the following result in \cite[Theorem 1.1]{DadToms:Z}:

\begin{thm}[Dadarlat--Toms] \label{thm:Dad-Toms}
The following conditions are equivalent for each unital separable \Cs{} $A$:
\begin{enumerate}
\item $\otmininfty A \cong \Big(\otmininfty A\Big) \otimes \cZ$. \vspace{.1cm}
\item There is a unital embedding of $\cZ$ into $\otmininfty A$. \vspace{.1cm}
\item $\otmininfty A$ contains unitally a  subhomogeneous \Cs{} without characters.
\end{enumerate}
\end{thm}

\noindent The dimension drop \Cs s
$$I(n,m) := \big\{f \in C([0,1],M_n \otimes M_m) \mid f(0) \in M_n \otimes \C, \; f(1) \in \C \otimes M_m\big\},$$
are subhomogeneous \Cs s for all pairs of integers $n,m \ge 1$; and $I(n,m)$ has no characters when $n,m\ge 2$. It was shown by Jiang and Su that that $I(n,m)$ embeds unitally into their algebra $\cZ$ if and only if $n$ and $m$ are relatively prime, see \cite{JiangSu:Z}, in fact $\cZ$ is the inductive limit of such algebras. 

%Combining Theorem \ref{thm:Dad-Toms} and  Corollary~\ref{cor:Abel} with the well-known fact, see eg.\ \cite[Theorem 7.2.2]{Ror:Encyc}, that (i) and (ii) in the theorem below are equivalent, we get:

\begin{thm} \label{thm:Z-absorption}
The following conditions are equivalent for each unital separable \Cs{} $A$:
\begin{enumerate}
\item $A \cong A \otimes \cZ$. \vspace{.1cm}
\item There is a unital embedding of $\cZ$ into $F(A)$. \vspace{.1cm}
\item $F(A)$ contains unitally a separable subhomogeneous \Cs{} without characters. 
\end{enumerate}
\end{thm}

\begin{proof}  It is well-known that (i) $\Leftrightarrow$ (ii) holds, see for example
\cite[Theorem 7.2.2]{Ror:Encyc}. 

(ii) $\Rightarrow$ (iii) follows from the fact that the  Jiang-Su algebra $\cZ$ contains the dimension drop \Cs{} $I(2,3)$. 

(iii) $\Rightarrow$ (ii). Let $D$ be a separable unital sub-\Cs{} of $F(A)$ which is subhomogenous and without characters. Then $D$ is nuclear, there is a unital \sh{} from $\bigotimes^\infty  D$ into $F(A)$ by Corollary~\ref{cor:Abel}, and there is a unital embedding of $\cZ$ into $\bigotimes^\infty  D$ by Theorem~\ref{thm:Dad-Toms}.
\end{proof}

\noindent Having proved Theorem \ref{thm:Dad-Toms} it was natural for Dadarlat and Toms to ask the following:

\begin{ques}[Dadarlat-Toms] \label{q3}
Let $D$ be a unital separable \Cs{} without characters. Does it follow that $\cZ$ embeds unitally into $\otmininfty D$?
\end{ques}

\noindent To provide an affirmative answer to Dadarlat and Toms' question, all we need to do is to embed unitally  into $\otmininfty D$ some subhomogeneous \Cs{} without characters, for example the dimension drop \Cs{} $I(2,3)$.

In Question~\ref{q3}  it is crucial that it is the same \Cs{} $D$ that is repeated infinitely many times, cf.\  the following theorem, \cite[Theorem 7.17]{RobRor:divisible},  by Robert and the first named author:

\begin{thm}[\cite{RobRor:divisible}]
There exist unital, simple, infinite dimensional, separable, nuclear \Cs s $D_1,D_2,D_3, \dots$ such that  $\cZ$ does not embed unitally into $\bigotimes_{n=1}^\infty \, D_n$. 
\end{thm}

\noindent The \Cs s in the theorem above are in fact AH-algebras, so they each contain subhomogeneous \Cs s without characters. However, they do not all contain \emph{the same} subhomogeneous \Cs{} without characters. 

\begin{rem}[The Cuntz semigroup and comparison of positive elements] \label{rem:Cuntz}
We remind the reader of the following few facts about the Cuntz semigroup that will be used in this paper. If $a,b$ are positive elements in $A \otimes \cK$, then write $a \precsim b$ if there is a sequence $\{x_k\}$ in $A \otimes \cK$ such that $x_k^*bx_k \to a$. Write $a \approx b$ if $a \precsim b$ and $b \precsim a$, and write $a \sim b$ if $a = x^*x$ and $b = xx^*$ for some $x \in A \otimes \cK$. We say that two positive elements are \emph{equivalent} if the latter relation holds between them. The Cuntz semigroup, $\Cu(A)$, of $A$ is defined to be the set of $\approx$-equivalence classes $\langle a \rangle$, where $a$ is a positive element in $A \otimes \cK$. The Cuntz relation $\precsim$ induces an order relation $\le$ on $\Cu(A)$, and addition in $\Cu(A)$ is given by orthogonal sum. Finally, one writes $\langle a \rangle \ll \langle b \rangle$ if $a \precsim (b-\ep)_+$ for some $\ep>0$. 
\end{rem}

%%%%%%%%%%%%%%%%%%%%%%%%%%%%%%%%%%%%%%%%%%%%%%%%%%%%
\section{Reformulations of the Dadarlat--Toms' question}

\noindent Prompted by Theorems~\ref{thm:Dad-Toms} and \ref{thm:Z-absorption} we ask the following:

\begin{ques} \label{q1}
Let $A$ be a unital separable \Cs. Does it follow that $A \cong A \otimes \cZ$ if and only if $F(A)$ has no characters?
\end{ques}

\noindent The "only if" part is trivially true, cf.\ Theorem~\ref{thm:Z-absorption}.
We show below that our Question~\ref{q1} very much is related to Question \ref{q3} of Dadarlat and Toms.

\begin{lemma} \label{lm:A-char}
If $A$ is a unital \Cs{} that admits a character, then $F(A)$ also admits a character.
\end{lemma}

\begin{proof}  Suppose that $\rho$ is a character on $A$. Then
$$\rho_\omega(\pi_\omega(x)) = \lim_\omega \rho(x_n), \qquad x = (x_n)_{n=1}^\infty \in \ell^\infty(A),$$
defines a character on $A_\omega$ (where $\pi_\omega \colon \ell^\infty(A) \to A_\omega$ is the canonical quotient map). The restriction of $\rho_\omega$ to $F(A) \subseteq A_\omega$ is then a character on $F(A)$.
\end{proof}

\noindent The converse to Lemma \ref{lm:A-char} is of course false, see Remark~\ref{rem:CFP}. Characterizations of unital \Cs s without characters were given in \cite{RobRor:divisible} as well as in \cite{Kir:abel}. We shall here give yet another, but related, description of such \Cs s. For each integer $n \ge 1$ consider the universal unital \Cs:
$$A(n,2) := \Big\{ a_1, \dots, a_n, b_1, \dots, b_n \; \big| \; \sum_{k=1}^n a_k^*a_k = 1, \; b_j^*a_j = 0, \; b_j^*b_j = a_j^*a_j, \, j=1, \dots n\Big\}.$$
In a similar way one can  define unital \Cs s $A(n,k)$ for each integer $k \ge 2$, but we shall not need these algebras here. Observe that $A(1,2)$ is the Cuntz-Toeplitz algebra $\mathcal{T}_2$. 

\begin{prop} \label{prop:A(n,2)}
\mbox{}
\begin{enumerate}
\item The \Cs{} $A(n,2)$ is unital, separable and has no characters. \vspace{.1cm}
\item There is a unital \sh{} $A(n,2) \to A(m,2)$ whenever $n \ge m$. \vspace{.1cm}
\item If $A$ is a unital \Cs, then $A$ has no characters if and only if there is a unital \sh{} $A(n,2) \to A$ for some integer $n \ge 1$.
\end{enumerate}
\end{prop}

\begin{proof} (i). The \Cs{} $A(n,2)$ is unital by definition, and separable because it is finitely generated. If $\rho$ is a character on $A(n,2)$, then 
$$1 = \rho\big(\sum_{k=1}^n a_k^*a_k\big) = \sum_{k=1}^n |\rho(a_k)|^2 = \sum_{k=1}^n |\rho(a_k)||\rho(b_k)| = \sum_{k=1}^n |\rho(b_k^*a_k)| = 0,$$
a contradiction.

(ii). By the universal properties of the \Cs s $A(n,2)$ and $A(m,2)$ it follows that there is a unital \sh{} $\varphi \colon A(n,2) \to A(m,2)$ given by  
$$\varphi\big(a_j^{(n)}\big) = \begin{cases} a_j^{(m)}, & j \le m,\\ 0, & j > m, \end{cases} \qquad 
\varphi\big(b_j^{(n)}\big) = \begin{cases} b_j^{(m)}, & j \le m,\\ 0, & j > m. \end{cases}$$

(iii). The "if" part follows from (i). Suppose that $A$ has no characters. It then follows from \cite[Corollary 5.4 (iii)]{RobRor:divisible} that there exist an integer $n \ge 1$ and \sh s $\psi_j \colon CM_2 \to A$, $j=1, \dots, n$, such that $\bigcup_{j=1}^n \psi_j(CM_2)$ is full in $A$. Here $CM_2 = C_0((0,1]) \otimes M_2$ is the cone over $M_2$. 

By repeating the $\psi_j$'s  (and increasing the number $n$), we can assume that there exist elements $d_1, \dots, d_n$ in $A$ such that 
$$1_A = \sum_{j=1}^n d_j^*\, \psi_j(\iota \otimes e_{11})\, d_j,$$
where $\iota \in C_0((0,1])$ denotes the (positive) function  $t \mapsto t$, and $e_{ij} \in M_2$, $i,j = 1,2$, are the standard matrix units. Put
$$a_j' = \psi_j(\iota^{1/2} \otimes e_{11})\, d_j, \qquad b_j' = \psi_j(\iota^{1/2} \otimes e_{21}) \, d_j, \quad j=1,\dots,n.$$
These elements are easily seen to satisfy the relations of the algebra $A(n,2)$. Hence there exists a unital \sh{} $A(n,2) \to A$ satisfying $a_j \mapsto a_j'$ and $b_j \mapsto b_j'$. 
\end{proof}

\begin{rem} \label{rem:cov-div} Let $A$ be a unital \Cs. The least number $n$ for which there is a unital \sh{} $A(n,2) \to A$ is related to the covering number $\mathrm{Cov}(A,2)$ from \cite{Kir:abel} as well as the weak divisibility number $\wDiv_2(A)$ from \cite{RobRor:divisible}. It is easy to see that $\wDiv_2(A) \le n$, and one can show that $n \le 3 \, \mathrm{Cov}(A,2)$. It was shown in \cite[Proposition 3.7]{RobRor:divisible} that $\mathrm{Cov}(A,2) \le \wDiv_2(A) \le 3 \,  \mathrm{Cov}(A,2)$. Combining these facts we get that
$$\mathrm{Cov}(A,2) \le \wDiv_2(A) \le n \le  3 \, \mathrm{Cov}(A,2).$$
\end{rem}

\begin{lemma} \label{lm:sep-characters}
Each unital \Cs{} without characters contains a unital \emph{separable} sub-\Cs{} without characters.
\end{lemma}

\begin{proof} This follows immediately from Proposition~\ref{prop:A(n,2)} (i) and (iii). 

One can also prove this claim directly as follows: Let $A$ be a unital \Cs{} without characters, and denote by $S(A)$ the weak-$^*$ compact set of its states. A state $\rho$ on $A$ is a character if and only $|\rho(u)|=1$ for all unitaries $u$ in $A$. Hence, $S(A)$ is covered by the family of open sets 
$$V_u := \{\rho \in S(A) : |\rho(u)| < 1\}, \qquad u \in U(A).$$
It follows that there exists a finite set $u_1,u_2, \dots, u_n$ of unitaries in $A$ such that $S(A)$ is covered by the corresponding open sets $V_{u_j}$, $ 1 \le j \le n$. Let $B$ be the separable unital sub-\Cs{} of $A$ generated by these unitaries. Then no state $\rho$ on $B$ can be a character, since its extension $\bar{\rho}$ to $A$ will belong to $V_{u_j}$ for some $j$, whence $|\rho(u_j)| = 
|\bar{\rho}(u_j)| < 1$. 
\end{proof}

\noindent Combining the results above with Corollary \ref{cor:Abel} we obtain the following dichotomy for the central sequence algebras:

\begin{prop} \label{prop:dichotomy}
Let $A$ be a unital separable \Cs. Then its central sequence algebra $F(A)$ either has a character or has no finite dimensional representation on a Hilbert space. In the latter case, there is a unital \sh{} from $\otmaxinfty A(n,2)$ into $F(A)$ for some $n \ge 1$. 
\end{prop}

\noindent We shall consider the following stronger version  of Question~\ref{q1}: 

\begin{ques} \label{q2}
Let $D$ be a unital separable \Cs{} without characters. Does it follow that $\cZ$ embeds unitally into $\otmaxinfty D$?
\end{ques}

\noindent To decide if Question~\ref{q2} has an affirmative answer, by Dadarlat and Toms' Theorem~\ref{thm:Dad-Toms}, one needs to show that whenever $D$ is a unital \Cs{} without characters, then $\otmaxinfty D$ contains unitally a  subhomogeneous \Cs{} without characters.

\begin{ques} \label{q2'}
Does $\otmaxinfty A(n,2)$ contain unitally a subhomogeneous \Cs{} without characters for each integer $n \ge 1$?
\end{ques}

\begin{prop} Questions \ref{q1}, \ref{q2}, and \ref{q2'} are equivalent.
\end{prop}

\begin{proof} It follows from Proposition \ref{prop:A(n,2)} and the remarks above that Questions~\ref{q2} and \ref{q2'} are equivalent. 

Suppose that Question \ref{q1} has an affirmative answer, and let $D$ be a unital separable \Cs{} without characters. There is a unital embedding of $D$ into $F(\otmaxinfty D)$, whence $F(\otmaxinfty D)$ has no characters.  It therefore follows that $\otmaxinfty D$ tensorially absorbs $\cZ$. In particular, $\cZ$ embeds unitally into  $\otmaxinfty D$. Hence Question \ref{q2} has an affirmative answer.

Suppose that Question \ref{q2} has an affirmative answer. Suppose that $A$ is a unital \Cs{} and that $F(A)$ has no characters. Then, by Lemma~\ref{lm:sep-characters}, there is a separable unital \Cs{} $D$ without characters that embeds unitally into $F(A)$. It follows from Corollary~\ref{cor:Abel}, and the assumption that Question \ref{q2} has an affirmative answer, that there are unital \sh s $\cZ \to \otmaxinfty D \to F(A)$. Hence $A \cong A \otimes \cZ$ by Theorem~\ref{thm:Z-absorption}.
\end{proof}

\noindent We proceed to relate the (original) Question \ref{q3} of Dadarlat--Toms to Question \ref{q1}. 

We remind the reader of the definition of the ideal $J(A)$ of the limit algebra $A_\omega$ associated with a unital \Cs{} $A$ and a free ultrafilter $\omega$ on $\N$. For each $p \ge 1$ and for each $\tau \in T(A)$, define semi-norms on $A$ as follows:
$$\|a\|_{p,\tau} = \tau((a^*a)^{p/2})^{1/p}, \qquad \|a\|_p = \sup_{\tau \in T(A)} \|a\|_{p,\tau}, \qquad a \in A.$$
If $a = \pi_\omega(a_1,a_2, \dots) \in A_\omega$, where $\pi_\omega \colon \ell^\infty(A) \to A_\omega$ is the quotient mapping, then set
$$\|a\|_{p,\omega} =  \lim_{n \to \omega} \|a_n\|_p.$$
We often write  $\|a\|_p$ instead of   $\|a\|_{p,\omega}$.

Let $J(A)$ be the closed two-sided ideal of $A_\omega$ consisting of all $a \in A_\omega$ such that $\|a\|_2 = 0$ (or, equivalenty, such that $\|a\|_p = 0$ for some $p \ge 1$). Note that $J(A) = A_\omega$ if and only if $T(A) = \emptyset$.

For each sequence $\{\tau_n\}_{n=1}^\infty$ of tracial states on $A$ we can associate a tracial state $\tau$ on $A_\omega$ by 
$$\tau\big(\pi_\omega(a_1,a_2,a_3, \dots)\big) = \lim_{n \to \omega} \tau_n(a_n), \qquad \{a_n\}_{n=1}^\infty \in \ell^\infty(A).$$
Let $T_\omega(A)$ denote the set of tracial states on $A_\omega$ that arise in this way. In particular, each tracial state $\tau$ on $A$ extends to a tracial state $\tau$ on $A_\omega$ by applying the construction above to the constant sequence $\{\tau\}_{n=1}^\infty$. 
Thus $T(A) \subseteq T_\omega(A) \subseteq T(A_\omega)$.

For each $\tau \in T(A_\omega)$ and for each $p \ge 1$ define a semi-norm on $A_\omega$ by $\|a\|_{p,\tau} = \tau\big((a^*a)^{p/2}\big)^{1/p}$. Let $J_\tau(A)$ denote the closed two-sided ideal of $A_\omega$ consisting of all $a \in A_\omega$ such that $\tau(a^*a) = 0$ (or, equivalently, such that $\|a\|_{p,\tau} = 0$ for some/all $p \ge 1$). 

\begin{lemma} \label{lm:separating}
Let $A$ be a unital separable \Cs{} with $T(A) \ne \emptyset$, and let $\omega$ be a free ultrafilter on $\N$. Then:
\begin{enumerate}
\item $\|a\|_{p,\omega} = \sup_{\tau \in T_\omega(A)} \|a\|_{p, \tau}$ for all $a \in A_\omega$. \vspace{.1cm}
\item $\displaystyle{J(A) = \bigcap_{\tau \in T_\omega(A)} J_\tau(A).}$
\item $A_\omega/J(A)$ and $F(A)/(F(A) \cap J(A))$ have separating families of traces.
\end{enumerate}
\end{lemma}

\begin{proof}  
(i). Write $a = \pi_\omega(a_1,a_2, \dots)$ with $\{a_n\} \in \ell^\infty(A)$. Let $\tau \in T_\omega(A)$ and represent $\tau$ by a sequence $\{\tau_n\}$ of tracial states on $A$. Then
$$\|a\|_{p,\tau} = \lim_\omega \|a_n\|_{p,\tau_n}.$$
Since $\|a_n\|_{p,\tau_n} \le \|a_n\|_p$, it follows that $\|a\|_{p,\tau}  \le \lim_\omega \|a_n\|_p = \|a\|_{p,\omega}$. Conversely, given $a \in A_\omega$ as above, we can for each natural number $n$ choose a tracial state $\tau_n$ on $A$ such that $\|a_n\|_{p,\tau_n} = \|a_n\|_p$. Let $\tau \in T_\omega(A)$ be the trace on $A_\omega$ associated with this sequence. Then $\|a\|_{p,\tau} = \lim_\omega \|a_n\|_p = \|a\|_{p,\omega}$.

(ii) follows immediately from (i). It follows from (ii) that $T_\omega(A)$ is a separating family of traces for $A_\omega/J(A)$; and hence also for the subalgebra $F(A)/(F(A) \cap J(A))$.
\end{proof}

\begin{rem} Ozawa proved in \cite{Ozawa:Dixmier} that $T_\omega(A)$ is weak$^*$ dense in $T(A_\omega)$ if $A$ is exact and $\cZ$-stable. For such \Cs s $A$ we get that
$$J(A) = \bigcap_{\tau \in T_\omega(A)} J_\tau(A)=\bigcap_{\tau \in T(A_\omega)} J_\tau(A),$$
and $\|a\|_{p,\omega} = \sup_{\tau \in T(A_\omega)} \|a\|_{p,\tau}$ for all $a \in A_\omega$. 
In particular, $J(A)$ is the smallest ideal in $A_\omega$ for which $A_\omega/J(A)$ has a separating family of traces. 

It is an easy consequence of  \cite[Theorem 1.4]{Robert:commutators} by Roberts that there exists a simple unital AH-algebra $A$ such that $T_\omega(A)$ is not weak$^*$ dense in $T(A_\omega)$. In other words, $A_\omega$ can have exotic traces that do not come from $A$ (not even close). 
\end{rem}

\noindent Fix a faithful tracial state $\tau$ on a \Cs{} $A$. 
Let $M$ be the type II$_1$-von Neumann algebra $\pi_\tau(A)''$, and let $M^\omega$ be the von Neumann central sequence algebra $\ell^\infty(M)/c_{\tau,\omega}(M)$, where $c_{\tau,\omega}(M)$ is the closed two-sided ideal in $\ell^\infty(M)$ consisting of all bounded sequences $\{a_n\}_{n=1}^\infty$ from $M$ such that $\lim_{n \to \omega} \|a_n\|_{\tau,2} = 0$. Then
$$A_\omega /J_\tau(A) \cong M^\omega, \qquad F(A)/(J_\tau(A) \cap F(A)) \cong M^\omega \cap M',$$
cf.\ \cite[Theorem 3.3]{KirRor:Central-sequence}.

The following result is well-known to experts. We include a sketch of the proof for the convenience of the reader.

\begin{prop} \label{prop:II_1-max-min}
Let $M$ be a type II$_1$-factor and let $N_1, N_2 \subseteq M$ be commuting sub-von Neumann algebras such that $(N_1 \cup N_2)'' = M$. It follows that the natural \sh{} 
$$N_1 \odot N_2 \to M, \qquad x \otimes y \mapsto xy, \qquad x \in N_1, \; y \in N_2,$$ 
extends to a \sh{} $N_1 \otmin N_2\to M$.
\end{prop}

\noindent Note that $N_1$ and $N_2$ necessarily are factors. 

\begin{proof} Let $\tau$ denote the tracial state on $M$. We may assume that $M$ acts on the Hilbert space $L^2(M,\tau)$. The map $x \otimes y \mapsto xy$, where $x \in N_1 \subseteq L^2(N_1,\tau)$, $y \in N_2 \subseteq L^2(N_2,\tau)$, and $xy \in M \subseteq L^2(M,\tau)$, extends to a unitary operator 
$$U \colon L^2(N_1,\tau) \otimes L^2(N_2,\tau) \to L^2(M,\tau).$$
The \sh{}
$$\xymatrix{N_1 \otmin N_2  \ar[r] &  B\big(L^2(N_1,\tau) \otimes L^2(N_2,\tau)\big) \ar[r]^-{\mathrm{Ad} \, U}  &B(L^2(M,\tau))}$$
is then the desired extension of the natural \sh{} $N_1 \odot N_2 \to M$.
\end{proof}

\begin{prop} \label{prop:C*-min-max}
Let $A$, $B$, and $D$ be unital \Cs{} and let $\varphi_0 \colon A \odot B \to D$ be a unital \sh.  Suppose that $D$ has a separating family of tracial states. Then $\varphi_0$ extends to a unital \sh{} $\varphi \colon A \otmin B \to D$. 
\end{prop}

\begin{proof} The \sh{} $\varphi_0$ extends to a unital \sh{} $\overline{\varphi} \colon A \otmax B \to D$. Upon replacing $D$ by the image of $\overline{\varphi}$ we may assume that $\overline{\varphi}$ is surjective. 

Let $I$ denote the kernel of the natural \sh{} $A \otmax B \to A\otmin B$. We must show that $\overline{\varphi}$ is zero on $I$. Suppose, to reach a contradiction, that $\overline{\varphi}(x) \ne 0$ for some $x \in I$ (that we can take to be positive). Then $\tau(\overline{\varphi}(x)) \ne 0$ for some trace $\tau$ on $D$; and hence also for some extremal trace $\tau$ on $D$. 

Let $(\pi_\tau,H_\tau)$ denote the GNS-representation of $D$ with respect to the trace $\tau$. Then $M = \pi_\tau(D)''$ is a II$_1$-factor (because $\tau$ is extremal). Put 
$$N_1 = (\pi_\tau \circ \overline{\varphi})(A \otimes 1_B)'', \qquad N_2 = (\pi_\tau \circ \overline{\varphi})(1_A \otimes B)''.$$
Then $N_1$ and $N_2$ are commuting sub-von Neumann algebras of $M$, and $M = (N_1 \cup N_2)''$. It follows from Proposition \ref{prop:II_1-max-min} that there is a unital \sh{} $\rho$ making the diagram 
$$\xymatrix{A \otmax B \ar[rr]^{\overline{\varphi}} \ar[d] && D \ar[d]^{\pi_\tau} & \\ A \otmin B \ar[r] &
 N_1 \otmin N_2 \ar[r]_-{\rho} & M}$$
commutative. This, however, leads to a contradiction, because commutativity of the diagram entails that $\pi_\tau \circ \overline{\varphi}$ is zero on $I$, whereas $\tau_M(\pi_\tau(\overline{\varphi}(x))) = \tau(\overline{\varphi}(x)) \ne 0$.
\end{proof}

\begin{prop} \label{prop:C*-min-max-inf}
Let $A_1,A_2, A_3, \dots$ be a sequence of unital \Cs s, let $D$ be a unital \Cs{} which has a separating family of tracial states, and let $\varphi_0 \colon \bigodot_{n \in \N} A_n \to D$ be a \sh. Then $\varphi_0$ extends to a unital \sh{} 
$$\varphi \colon \textstyle{\bigotimes^{\mathrm{min}}_{n \in \N}}  \, A_n \to D.$$
\end{prop}

\begin{proof} Consider the natural \sh{}
$$\pi \colon \textstyle{\bigotimes^{\mathrm{max}}_{n \in \N}}  \, A_n  \to \textstyle{\bigotimes^{\mathrm{min}}_{n \in \N}} \,  A_n,$$ and the natural extension $\overline{\varphi} \colon \bigotimes^{\mathrm{max}}_{n \in \N}  \, A_n  \to D$ of $\varphi_0$. For each integer $N \ge 2$ consider also the natural \sh{}
$$\pi_N \colon \textstyle{\bigotimes^{\mathrm{max}}_{1\le n \le N}}  \; A_n  \to \textstyle{\bigotimes^{\mathrm{min}}_{1 \le n \le N}}  \; A_n,$$ and the restriction $\overline{\varphi}_N \colon \textstyle{\bigotimes^{\mathrm{max}}_{1 \le n \le N} } \; A_n  \to D$ of $\overline{\varphi}$.
We must show that $\Ker(\pi) \subseteq \Ker(\overline{\varphi})$. It follows from repeated applications of Proposition \ref{prop:C*-min-max} that $\Ker(\pi_N) \subseteq \Ker(\overline{\varphi}_N)$ for all $N$. This verifies the claim because 
$$\Ker(\pi) = \overline{\bigcup_{N=2}^\infty \Ker(\pi_N)}, \qquad \Ker(\overline{\varphi}) = \overline{\bigcup_{N=2}^\infty \Ker(\overline{\varphi}_N)}.$$
\end{proof}

\noindent The proposition below is an analog of Corollary~\ref{cor:Abel}:

\begin{prop} \label{prop:quotient} Let $A$ be a unital separable \Cs{} for which $T(A) \ne \emptyset$, and let $D$ be a  a unital separable sub-\Cs{} of $F(A)$. Then there  is a unital \sh{} 
$$
\otmininfty D \to F(A)/(F(A) \cap J(A)).
$$
In particular, if $F(A)$ has no characters, then there is such a unital \sh{} for some unital separable \Cs{} $D$ without characters. (This \Cs{} $D$ can further be taken to be $A(n,2)$ for some $n \ge 1$.)
\end{prop}

\begin{proof} 
It follows from Corollary~\ref{cor:Abel} that there is a unital \sh{}
$\otmaxinfty D \to F(A)$, and hence a unital \sh{} $$\otmaxinfty D \to F(A)/(F(A) \cap J(A)).$$ As $F(A)/(F(A) \cap J(A))$ has a separating family of tracial states (by Lemma~\ref{lm:separating}), the existence of the desired unital \sh{} follows from Proposition~\ref{prop:C*-min-max-inf}.

The second part of the proposition follows from Lemma~\ref{lm:sep-characters} and Proposition~\ref{prop:A(n,2)}.
\end{proof}

\noindent We remind the reader of property (SI) of Matui and Sato (see for example \cite[Definition~4.1]{MatuiSato:Z}) in the formulation of \cite{KirRor:Central-sequence}: A unital \Cs{} $A$ is said to have property (SI) if for all positive contractions $e,f \in F(A)$ with $e \in J(A)$ and $\sup_k \|1-f^k\|_2 < 1$ there exists $s \in F(A)$ with $fs = s$ and $s^*s = e$. This implies that $e \precsim f$ in $F(A)$ (see Remark~\ref{rem:Cuntz}) and moreover, that $e \precsim (f-1/2)_+$. 

Every simple, unital, separable, stably finite, nuclear \Cs{} with the "local weak comparison property" (see \cite{KirRor:Central-sequence}) has property (SI) by \cite{MatuiSato:Z} and \cite{KirRor:Central-sequence}. 

Property (SI) implies that certain liftings from $F(A)/(F(A) \cap J(A))$ to $F(A)$ are possible (as proved for example in \cite{MatuiSato:Z} and in \cite[Proposition 5.12]{KirRor:Central-sequence}). We need here a stronger lifting result (Proposition \ref{prop:5.12'} below). The next lemma about the dimension drop \Cs s $I(k,k+1)$ is an elaboration of known results:

\begin{lemma} \label{lm:dimensiondrop}
Let $k \ge 2$ be an integer.
\begin{enumerate}
\item There are positive contractions $a_1, \dots, a_k$ in $I(k,k+1)$, which are pairwise orthogonal and equivalent, such that \vspace{.1cm}
\begin{itemize}
\item[(a)] $a_0 := 1-(a_1+ \cdots +a_k) = t^*(a_1-1/2)_+t$ for some $t \in I(k,k+1)$, \vspace{.1cm}
\item[(b)] $\tau(a_1^n) \ge \frac{1}{k+1}$ for all tracial states $\tau$ on $I(k,k+1)$ and for all $n \ge 1$. 
\end{itemize} \vspace{.1cm}
\item If $A$ and $B$  are unital \Cs s, if $\pi \colon A \to B$ is a surjective unital \sh, and if $b_0,b_1, \dots, b_k$ are positive contractions in $B$ such that $b_1, \dots, b_k$ are pairwise orthogonal and equivalent, $b_0+b_1+ \cdots + b_k = 1_B$, and $b_0 = t(b_1-1/2)_+t^*$ for some $t \in B$, then there exist positive contractions $a_0,a_1, \dots, a_k$ in $A$ such that $a_1, \dots a_k$ are pairwise orthogonal and equivalent, and 
$$a_0 \precsim (a_1-1/2)_+, \qquad \pi(a_j) = b_j,  \quad j=0,1, \dots k.$$
\end{enumerate}
\end{lemma}

\begin{proof}
(i). By the universal property for $I(k,k+1)$, cf.\ \cite[Proposition 5.1]{RorWin:Z}, there are positive contractions $b_1, \dots, b_k$ in $I(k,k+1)$, which are pairwise orthogonal and equivalent, such that 
$$b_0:= 1-(b_1+ \cdots + b_k) \precsim (b_1-\ep)_+$$
for some $\ep >0$. Choose $\eta \in (0,\ep)$ such that
$$\frac{k(1-\eta)-1}{k^2(1-\eta)} \le \frac{1}{k+1}.$$
Consider the continuous functions $g_\eta, h_\eta \colon [0,1] \to [0,1]$ given by 
$$g(t) = \begin{cases} \eta^{-1}t, & 0 \le t \le \eta,\\ 1, & t \ge \eta, \end{cases} \qquad h(t) = \begin{cases} 0, & 0 \le t \le 1-\eta, \\ \eta^{-1}(t-1+\eta), & t \ge 1-\eta.\end{cases}$$
Observe that $h(t) = 1-g(1-t)$. Put $a_j = g(b_j)$ for $j=1, 2, \dots, k$, and put $a_0 = 1 - (a_1+\cdots +a_k)= h(b_0)$. It is clear that $a_1, \dots, a_k$ are pairwise orthogonal and  equivalent. Let us check that (a) and (b) hold.

(a). Since $0<\eta < \ep$ there are elements $t_1, t_2 \in I(k,k+1)$ such that $t_1^*(a_1-1/2)_+t_1 = (b_1-\ep)_+$ and $t_2^*(b_1-\ep)_+t_2 = h(b_0) = a_0$. Hence $t = t_1t_2$ is as  desired. 

(b). Let $\tau$ be a tracial state on $I(k,k+1)$ and embed $I(k,k+1)$ into a finite von Neumann algebra $M$ such that $\tau$ extends to a trace on $M$. Set $p_j = 1_{[\eta,1]}(b_j)$ for $j=1,2, \dots, k$. Then $p_1 \sim p_2 \sim \cdots \sim p_k$ in $M$, 
$$p:= 1_{[\eta,1]}(b_1+b_2 + \cdots + b_k) = p_1+p_2+ \cdots +p_k,$$
and $(b_1+ \cdots + b_k)(1-p) \le \eta (1-p)$. It follows that $b_0(1-p) \ge (1-\eta)(1-p)$. Thus
\begin{eqnarray*}
\tau(1-p) & \le & (1-\eta)^{-1}\tau(b_0(1-p)) \le  (1-\eta)^{-1} d_\tau(b_0) \\ & \le & (1-\eta)^{-1}d_\tau((b_1-\ep)_+) \le (1-\eta)^{-1} k^{-1},
\end{eqnarray*}
where $d_\tau$ is the dimension function associated with $\tau$. Hence, for each $n \ge 1$,
$$\tau(a_1^n) \ge \tau(p_1) = k^{-1}\tau(p) \ge k^{-1}(1-  (1-\eta)^{-1} k^{-1}) \ge (k+1)^{-1}$$
as desired.

(ii). 
A standard trick, using that $C_0((0,1]) \otimes M_k$ is projective and that $b_1, \dots, b_k$ are the images of elements of the form $\iota \otimes e_{jj}$ under a \sh{} from $C_0((0,1]) \otimes M_k$  into $B$, shows that the elements $b_1, \dots, b_k$ lift to positive, pairwise orthogonal and  equivalent, contractions $a_1, \dots, a_k$ in $A$. Lift $t \in B$ to an element $s \in A$ and put $a_0 =s^*(a_1-1/2)_+s$. It is now clear that $a_0, a_1, \dots, a_k$ have the desired properties.
\end{proof}

\begin{prop} \label{prop:5.12'} The following conditions are equivalent for any  separable, simple, unital \Cs{} $A$ with $T(A) \ne \emptyset$ and which has property (SI):
\begin{enumerate}
\item There is a unital embedding $\cZ \to F(A)$. \vspace{.1cm}
\item There is a unital embedding $\cZ \to F(A)/(F(A) \cap J(A))$. \vspace{.1cm}
\item There is a unital \sh{} $I(4,5) \to F(A)/(F(A) \cap J(A))$.
\end{enumerate}
\end{prop}

\begin{proof} It is trivial that (i) $\Rightarrow$ (ii); and (ii) $\Rightarrow$ (iii) holds because $I(4,5)$ embeds unitally into $\cZ$. 

To prove that (iii) $\Rightarrow$ (i) it sufficies to show that if (iii) holds, then there is a unital \sh{} $I(2,3) \to F(A)$, cf.\ Theorem~\ref{thm:Z-absorption}. By Lemma~\ref{lm:dimensiondrop} (i) and (ii) there are positive contractions $a_0,a_1,a_2,a_3,a_4$ in $F(A)$ satisfying:
\begin{itemize}
\item $a_1,a_2,a_3,a_4$ are pairwise orthogonal and pairwise equivalent,
\item $a_0  \precsim (a_1-1/2)_+$, 
\item $\tau(a_j^n) \ge 1/5$ for $j=1,\dots, 4$, for all  $\tau \in T(A)$, and for all integers $n \ge 1$,
\item $e:=1-(a_0+a_1+ \cdots + a_4) \in J(A)$. 
\end{itemize}
Put $c = 1-(a_1+ \cdots + a_4)$. 
It follows from the fact that $J(A)$ is a $\sigma$-ideal in $A_\omega$ (cf.\ \cite[Definition 4.4 and Lemma 4.6]{KirRor:Central-sequence}) that there is a positive contraction $g$ in $F(A) \cap J(A)$ such that $ge=eg=e$ and $ga_j=a_jg$ for all $j$. Now,
$$c = gc+ (1-g)c = gc + (1-g)a_0 \precsim g \oplus a_0 \precsim g \oplus (a_1-1/2)_+.$$
By property (SI) it follows that there is $s \in F(A)$ such that $a_2s = s$ and $s^*s = g$. In particular, $g \precsim (a_2-1/2)_+$. 

Put $b_1 = a_1+a_2$ and $b_2 = a_3+a_4$. Then $b_1$ and $b_2$ are positive, pairwise orthogonal and equivalent contractions, and
$$1 - (b_1+b_2) = c \precsim g \oplus (a_1-1/2)_+ \precsim (a_1-1/2)_+ \oplus (a_2 -1/2)_+ \precsim (b_1-1/2)_+.$$
We now get the \sh{} $I(2,3) \to F(A)$ from \cite[Proposition 5.1]{RorWin:Z}. 
\end{proof}

\begin{prop} \label{prop:char-SI}
 Suppose that Question \ref{q3} has an affirmative answer. Then Question \ref{q1} has an affirmative answer for each separable, simple, unital \Cs{} $A$ with at least one tracial state and with property (SI).
\end{prop}

\noindent In other words, if Question \ref{q3} has an affirmative answer and if $A$ is a separable unital simple \Cs{} which admits a tracial state and has property (SI), then $A \cong A \otimes \cZ$ if and only if $F(A)$ has no characters. 

\begin{proof} Suppose that $A$ is a unital separable \Cs{} such that $F(A)$ has no characters. Then, by Proposition~\ref{prop:quotient}, there is a unital \sh{} 
$$\otmininfty D \to F(A)/(F(A) \cap J(A))$$
 for some separable unital \Cs{} $D$ without characters. If Question \ref{q3} has an affirmative answer, then Jiang-Su algebra $\cZ$ embeds unitally into $\otmininfty D$, and hence also into $F(A)/(F(A) \cap J(A))$. The conclusion now follows from Proposition~\ref{prop:5.12'} above. 
\end{proof}

\noindent We saw above that, in the presence of property (SI), $\cZ$ embeds unitally into $F(A)$ if it embeds unitally into $F(A) / (F(A) \cap J(A))$. The property of having no characters similarly lifts from $F(A) / (F(A) \cap J(A))$ to $F(A)$:

\begin{prop} \label{prop:char-equiv}
Let $A$ be a unital \Cs{} with property (SI) and for which $T(A) \ne \emptyset$. 
Then $F(A)$ has no characters if and only if  $F(A)/(F(A) \cap J(A))$ has no characters.
\end{prop}

\begin{proof} Note that $F(A) \cap J(A)$ is a proper ideal in $F(A)$ because $T(A) \ne \emptyset$, so $F(A)/(F(A) \cap J(A))$ is non-zero. Any character on the quotient $F(A)/(F(A) \cap J(A))$  lifts to a character on $F(A)$ by composition with the quotient map. 

To prove the "if"-part, suppose, to reach a contradiction, that $\rho$ is a character on $F(A)$ and that $F(A)/(F(A) \cap J(A))$ has no characters. Then $\mathrm{Ker}(\rho)$ cannot be contained in 
$F(A) \cap J(A)$, whence
$F(A) = \mathrm{Ker}(\rho) + F(A) \cap J(A)$. We can therefore find a positive contraction $e \in F(A) \cap J(A)$ such that $\rho(1-e)=0$. As $\|1-(1-e)^n\|_{2} = 0$ for all $n \ge 1$, property (SI) (the version given in \cite[Definition 2.6]{KirRor:Central-sequence}) gives an $s \in F(A)$ such that $(1-e)s = s$ and $s^*s = e$. Hence $\rho(s)=0$, so $\rho(e) = 0$, a contradiction.
\end{proof}
%Comment about the existence of $e$ above: If $A$ is a unital \Cs{} and if $I$ and $J$ are proper ideals in $A$ such that $A = I+J$, then there exists a positive contraction $e \in J$ such that $1-e \in I$. To see this let $\pi \colon A \to A/I$ denote the quotient mapping. Then $\pi(J) = \pi(A) = A/I$. We can therefore lift $\pi(1)$ to a positive contraction $e \in J$, and $1-e \in I$.

\noindent We end this section by giving an alternative definition of property (SI) using Lemma~\ref{lm:separating}. 

\begin{prop} \label{prop:SI} A unital separabel \Cs{} has property (SI) if and only if the following holds for all positive elements $a$ and $b$ in $F(A)$ and all $\delta >0$:
$$ \Big(\forall \tau \in T_\omega(A) : \tau(a) = 0 \; \text{and} \; \tau(b) \ge \delta\Big) \implies a \precsim b \; \text{in} \; F(A).$$
\end{prop}

\begin{proof} As in the proof of Lemma \ref{lm:separating} we see that 
\begin{equation} \tag{$\dagger$}
\|e\|_1 = \sup_{\tau \in T_\omega(A)} \tau(e)
\end{equation}
for all positive elements $e \in A_\omega$. 

Suppose first that $A$ has property (SI), and let $a,b \in A$ and $\delta >0$ be such that $\tau(a) =0$ and $\tau(b) \ge \delta$ for all $\tau \in T_\omega(A)$. We may assume that $a$ and $b$ are contractions (possibly upon changing $\delta$). Note that  $a$ belongs to $J(A)$ by Lemma \ref{lm:separating}.  Let $h \colon \R^+ \to [0,1]$ be a continuous function such that $h(0)=0$ and $h(t) = 1$ for all $t \ge \delta/2$. Put $f = h(b)$. Then $\tau(f^n) \ge \delta/2$ for all $n \ge 1$ and for all $\tau \in T_\omega(A)$. Hence $\|1-f^n\|_1 \le 1-\delta/2$ for all $n \ge 1$ by ($\dagger$). Since $A$ is assumed to have property (SI) there is $s \in F(A)$ such that $fs=s$ and $s^*s = a$. In particular, $a = s^*fs$, so $a \precsim f \precsim b$.

Suppose next that $A$ satisfies the condition of the proposition. Let $e,f$ be positive contractions in $F(A)$ satisfying $e \in J(A)$ and $\|1-f^n\|_1 \le 1-\delta$ for some $\delta >0$ and for all $n \ge 1$. Then $\tau(e) = 0$ for all $\tau \in T_\omega(A)$ by Lemma \ref{lm:separating}. Define a sequence $\{g_n\}$ of continuous functions on $[0,1]$ satisfying $g_n g_{n+1} = g_{n+1}$, $g_n(1)=1$, and $g_n|_{[0,1-n^{-1}]} \equiv 0$.  Then $\tau(g_n(f)) \ge \delta$ for all $\tau \in T_\omega(A)$ and for all $n$. By assumption, this implies that $e \precsim g_n(f)$, so there exists $t_n \in F(A)$ such that $\|t_n^*g_n(f)t_n-e\| \le 1/n$. Put $s_n = g_n(f)^{1/2}t_n$. Then $\|s_n\| \le 2$,  $\|s_n^*s_n - e\| \le 1/n$ and
$$\|(1-f)s_n\| = \|(1-f)g_{n-1}(f)s_n\| \le \|(1-f)g_{n-1}(f)\|\|s_n\| \le 2/(n-1)$$
for all $n$.
One can now use the "$\ep$-test" (see \cite[Lemma 3.1]{KirRor:Central-sequence}) to find $s \in F(A)$ such that $s^*s = e$ and $(1-f)s = 0$. 
\end{proof}
%%%%%%%%%%%%%%%%%%%%%%%%%%%%%%%%%%%%%%%%%%%%%%%%%%%%%%%%%%
\section{The central sequence algebra and the Corona Factorization Property}

\noindent A \Cs{} $A$ is said to have the \emph{Corona Factorization Property} if every full projection in the multiplier algebra of $A \otimes \cK$ is properly infinite. If every closed two-sided ideal of $A$ has the Corona Factorization Property, then we say that $A$ has the \emph{strong Corona Factorization Property}.  The Corona Factorization Property was studied by Elliott and Kucerovsky in \cite{EllKuc:Voiculescu} in order to obtain Voiculescu type absorption results for extensions of \Cs s. 

It was shown in \cite[Theorem 5.13]{OPR:Cuntz} that a separable \Cs{} $A$ has the strong Corona Factorization Property if and only if its Cuntz semigroup has the so-called strong Corona Factorization Property for semigroups, cf.\  \cite[Definition~2.12]{OPR:Cuntz}: For every $x',x,y_1,y_2,y_3, \dots$ in $\Cu(A)$ and $m \in \N$ such that
$x' \ll x$ and $x \le my_n$ in $\Cu(A)$ for all $n \ge 1$, there exists $k \ge 1$ such that $x' \le y_1+y_2+ \cdots + y_k$ in $\Cu(A)$. The (strong) Corona Factorization Property can therefore be viewed as a weak comparability property for $\Cu(A)$. (See Remark~\ref{rem:Cuntz} for the definition of the Cuntz semigroup.)

It was shown in \cite[Proposition 6.3]{RobRor:divisible} that $\bigotimes^\infty D$ has the strong Corona Factorization Property whenever $D$ is a unital \Cs{} without characters. (The argument works for any tensor product, for example the maximal one.) We use this result, along with Corollary~\ref{cor:Abel}, to show that any unital separable \Cs{} $A$ has the strong Corona Factorization Property if $F(A)$ has no characters. 

We need two lemmas, the first of which says that 
whenever $P$ is an intermediate \Cs{} between $A$ and $A_\omega$, then the map $\Cu(A) \to \Cu(P)$, induced by the inclusion $A \subseteq P$,  is an order inclusion.

\begin{lemma} \label{lm:A-D-comparison}
Let $A$ be a \Cs, let $\omega$ be a free filter on $\N$, and let $P$ be a \Cs{} such that $A \subseteq P \subseteq A_\omega$. 
\begin{enumerate}
\item If $x,y \in \Cu(A)$, then $x \le y$ in $\Cu(A)$ if and only if $x \le y$ in $\Cu(P)$, \vspace{.1cm}
\item If $x,x' \in \Cu(A)$, then $x' \ll x$ in $\Cu(A)$ if and only if $x' \ll x$ in $\Cu(P)$.
\end{enumerate}
\end{lemma}

\begin{proof} For each integer $n \ge 1$, we view  $M_n(A)$ and $M_n(P)$ as being hereditary sub-\Cs s of $A \otimes \cK$ and $P \otimes \cK$, respectively. Let $a,a',b$ be positive elements in $A \otimes \cK$ representing $x$, $x'$, and $y$, respectively.

(i). The "if"-part is clear. Suppose that $x \le y$ in $\Cu(P)$ and let $\ep >0$ be given. Then there exists $r \in P \otimes \cK$ such that $\|r^*br - a\| < \ep$. Observe that $$A \otimes \cK \; \subseteq \; P \otimes \cK \; \subseteq \; A_\omega \otimes \cK \; \subseteq \; (A \otimes \cK)_\omega.$$ Write
$$r = \pi_\omega(r_1, r_2,r_3, \dots),$$
where $r_k \in A \otimes \cK$ for each $k$. Then
$$\ep > \|r^*br - a\| = \limsup_\omega \|r_k^* b r_k - a\|.$$
It follows that $\|r_k^*br_k-a\|< \ep$ for some $k$. As $\ep>0$ was arbitrary this proves that $a \precsim b$ in $A \otimes \cK$; and hence $x \le y$ in $\Cu(A)$.

(ii). Observe that $x' \ll x$ (in $\Cu(A)$ or in $\Cu(P)$) if and only if $a' \precsim (a-\ep)_+$ (in $A \otimes \cK$ or in $P \otimes \cK$) for some $\ep >0$. Hence (ii) follows from (i).
\end{proof}

\begin{lemma} \label{prop:A-D-CFP}
Let $A$ be a separable \Cs, let $\omega$ be a free filter on $\N$, and let $P$ be a separable \Cs{} such that $A \subseteq P \subseteq A_\omega$. Then $A$  has the strong Corona Factorization Property if $P$ does.
\end{lemma}

\begin{proof}  We verify that $\Cu(A)$ has the strong Corona Factorization Property for semigroups. Accordingly, suppose that $x',x,y_1,y_2,y_3, \dots$ in $\Cu(A)$ and $m \in \N$ are given such that
$x' \ll x$ and $x \le my_n$ in $\Cu(A)$ for all $n \ge 1$. Since $P$ is assumed to have the strong Corona Factorization Property we know that there exists $k \ge 1$ such that $x' \le y_1+y_2+ \cdots + y_k$ in $\Cu(P)$. But then $x' \le y_1+y_2+ \cdots + y_k$ in $\Cu(A)$ by Lemma~\ref{lm:A-D-comparison}~(i). 
\end{proof}

\noindent  It is shown by Kucerovsky and Ng in \cite[Theorem 3.1]{KucNg:CFP} that the quotient of any separable \Cs{} with the Corona Factorization Property again has the Corona Factorization Property. It follows from this result that  the quotient of any separable \Cs{} with the strong Corona Factorization Property again has the strong Corona Factorization Property. 
Indeed, suppose that $A$ has the strong Corona Factorization property, that $B$ is a quotient of $A$, and that $\pi \colon A \to B$ is the quotient mapping. Let $I$ be a closed two-sided ideal of $B$, and let $J= \pi^{-1}(I)$. Then $J$ has the Corona Factorization Property, since it is a closed two-sided ideal in $A$, and  hence $I = \pi(J)$ has the Corona Factorization Property, because $I$ is a quotient of $J$. 

\begin{thm} \label{thm:CFP} Let $A$ be a unital separable \Cs{} such that the central sequence algebra $F(A)$ has no characters. Then $A$ has the strong Corona Factorization Property.
\end{thm}

\begin{proof} First use Lemma~\ref{lm:sep-characters} to find a separable unital sub-\Cs{} $D$ of $F(A)$ without characters. Then use Corollary~\ref{cor:Abel} to find a unital \sh{} 
$$\varphi \colon A \otmax \Big(\otmaxinfty D \Big) \to A_\omega$$
so that $\varphi(a \otimes 1) = a$ for all $a \in A$. Let $P$ be the image of $\varphi$. Then $A \subseteq P \subseteq A_\omega$, and $P$ is isomorphic to a quotient of $A \otmax \big(\otmaxinfty D \big)$. It was shown in \cite[Proposition~6.3]{RobRor:divisible} that $A \otmax \big(\otmaxinfty D \big)$ has the strong Corona Factorization Property (use Lemma~\ref{lm:A-char} to see that $A$ has no characters).  By the result of Kucerovsky and Ng mentioned above, we can conclude that $P$ has the strong Corona Factorization Property. 
It finally follows from Lemma~\ref{prop:A-D-CFP} that $A$ has the strong Corona Factorization Property.
\end{proof}

\noindent The contrapositive of Theorem~\ref{thm:CFP} is perhaps more interesting: If $A$ is a unital separable \Cs{} which does not have the (strong) Corona Factorization Property, then the central sequence algebra $F(A)$ has a character. In the remark below we use this to give  examples of separable nuclear \Cs s whose central sequence algebra has a character. 

\begin{rems} \label{rem:CFP}

 (i). The example in \cite{Ror:simple} of a simple separable nuclear \Cs{} $W$ with a finite and an infinite projection fails to have the Corona Factorization Property. Thus $F(W)$ has a character. 
This fact is not mentioned explicitly in \cite{Ror:simple}, but it follows from its construction.  Indeed, inspection shows that $W$  contains projections $p,q_0,q_1,q_2, \dots$ such that $p \nprecsim q_0 \oplus q_1 \oplus \cdots \oplus q_n$ while $p \precsim q_n \oplus q_n$ for all $n \ge 0$. Taking $x=x' = \langle p \rangle$, $y_j = \langle q_j \rangle$ in the Cuntz semigroup and $m=2$, we see that the Cuntz semigroup does not have the strong Corona Factorization Property.

(ii). It is shown in \cite[Corollary 5.16]{OPR:CFP-RR0} that every separable simple \Cs{} of real rank zero with the Corona Factorization Property is either stably finite or purely infinite. We do not know if this also holds for general separable simple \Cs{} (possibly not of real rank zero). Hence we do not know if any separable simple \Cs{} which contains a finite and an infinite projection automatically will fail to have the Corona Factorization Property. See also Proposition~\ref{prop:sf-pi} below.

(iii). The example in \cite{Ror:sns} of a simple non-stable AH-algebra $A$, such that some matrix algebra over $A$ is stable, must fail to have the Corona Factorization Property, cf.\ \cite[Corollary 4.3]{KucNg:CFP}. A unital corner $B$ of $A$ will serve as an example of a unital stably finite separable nuclear simple \Cs{} without the Corona Factorization Property. Thus $F(B)$ has a character.

(iv). Consider an example (as in (iii) above) of a simple, separable, unital, nuclear, stably finite \Cs{} $A$ which does not have the Corona Factorization Property. Then $F(A)$ has a character and also a  quotient isomorphic to a hyperfinite II$_1$-von Neumann algebra (by \cite[Lemma 2.1]{Sato:Jiang-Su}, see also \cite[Theorem 3.3]{KirRor:Central-sequence}). In particular, $F(A)$ is non-abelian. This contrasts the situation for II$_1$-factors, where McDuff proved that the von Neumann central sequence algebra either is abelian or a II$_1$-von Neumann algebra.

(v). Let us finally note that not all unital separable simple \Cs s with the Corona Factorization Property are $\cZ$-absorbing. Indeed, Kucerovsky and Ng produced in \cite{KucNg:JOT} an example of a unital separable simple \Cs{} with the Corona Factorization Property whose $K_0$-group has perforation. Hence it cannot absorb the Jiang-Su algebra, cf.\ \cite{GongJiangSu:Z}. We do not know if the central sequence algebra of such a \Cs{} has a character. 
\end{rems}

\noindent The lemma below is an easy consequence of associativity of the maximal tensor product. 

\begin{lemma} \label{lm:associative} Let $A,B,N$ be \Cs s with $N$ is nuclear. Then
$$(A \otmax B) \otmin N \; \cong \; A \otmax (B \otmin N).$$
\end{lemma}

\begin{ex} Let $W$ be the (nuclear) \Cs{} from Remark~\ref{rem:CFP} (i) and set $A = C^*_\mathrm{red}(\mathbb{F}_2) \otimes W$. Then $A$ is a simple, unital, separable, exact, purely infinite \Cs{} which does not absorb tensorially any non-elementary nuclear \Cs. In particular, $A$ does not absorb the Jiang-Su algebra nor the Cuntz algebra $\mathcal{O}_\infty$, and $F(A)$ does not contain any unital subhomogeneous without characters. 

Let us verify that $A$ has the stipulated properties. 
Simplicity of $A$ follows from Takesaki's theorem (because $W$ and $C^*_\mathrm{red}(\mathbb{F}_2)$ are simple); and $A$ is exact because both $W$ and $C^*_\mathrm{red}(\mathbb{F}_2)$ are exact. Since $A$ by construction is non-prime and not stably finite it follows from \cite[Theorem 4.1.10 (ii)]{Ror:Encyc} (a result of the first named author) that $A$ is purely infinite. 

Let us show that $A$ cannot be isomorphic $A \otimes B$ for any non-elementary nuclear \Cs{} $B$. Suppose, to reach a contradiction, that $A \cong A \otimes B$ for such a \Cs{} $B$. Then $B$ must be unital, separable, exact and simple. Applying \cite[Theorem~4.1.10~(ii)]{Ror:Encyc} again we see that $W \otimes B$ is purely infinite. Hence $W \otimes B$ is simple, separable, unital, nuclear and purely infinite, and therefore $W \otimes B \cong W \otimes B \otimes \mathcal{O}_\infty$ by \cite{KirPhi:classI}. As
$$A \cong A \otimes B \cong C^*_\mathrm{red}(\mathbb{F}_2) \otimes (W \otimes B),$$
we conclude that $A \cong A \otimes {\mathcal{O}}_\infty$. Using this identity and Lemma \ref{lm:associative} above twice we get
$$\big(C^*_\mathrm{red}(\mathbb{F}_2)  \otmax A\big) \otmin {\mathcal{O}}_\infty \cong C^*_\mathrm{red}(\mathbb{F}_2) \otmax A \cong \big(C^*_\mathrm{red}(\mathbb{F}_2)  \otmax C^*_\mathrm{red}(\mathbb{F}_2) \big) \otmin W,$$
so the \Cs{} on the right-hand side is (strongly) purely infinite. 

Akemann and Ostrand proved in \cite{AkeOst:F_2-tensor} that the \Cs{} of compact operators, $\mathcal{K}(\ell^2(\mathbb{F}_2))$, is contained in the image of the
 "left-right" regular representation of $C^*_\mathrm{red}(\mathbb{F}_2)  \otmax C^*_\mathrm{red}(\mathbb{F}_2)$ on $\ell^2(\mathbb{F}_2)$. Hence $\mathcal{K}(\ell^2(\mathbb{F}_2))$ is isomorphic to $J/I$ for some closed two-sided ideals $I \subset J$ in  $C^*_\mathrm{red}(\mathbb{F}_2)  \otmax C^*_\mathrm{red}(\mathbb{F}_2)$. Now, $I \otimes W \subset J \otimes W$ are closed two-sided ideals in the purely infinite \Cs{}  $\big(C^*_\mathrm{red}(\mathbb{F}_2)  \otmax C^*_\mathrm{red}(\mathbb{F}_2) \big) \otmin W$. Being purely infinite passes to ideals and to quotients, so 
$$\mathcal{K}(\ell^2(\mathbb{F}_2)) \otimes W \cong (J \otimes W) / (I \otimes W)$$
is purely infinite. This contradicts the fact that $W$ has a non-zero finite projection. We  conclude that $A$ does not absorb tensorially any non-elementary  nuclear \Cs.

We know from Remark~\ref{rem:CFP} (i) that $F(W)$ has a character. The von Neumann central sequence algebra, $\mathcal{L}(\mathbb{F}_2)^\omega \cap \mathcal{L}(\mathbb{F}_2)'$, is abelian (and hence has a character) because $\mathcal{L}(\mathbb{F}_2)$ is not a McDuff factor. Moreover, it is a quotient of $F(C^*_\mathrm{red}(\mathbb{F}_2))$, cf.\ \cite[Theorem~3.3]{KirRor:Central-sequence}, so $F(C^*_\mathrm{red}(\mathbb{F}_2))$  also has a character. In other words, $A = C^*_\mathrm{red}(\mathbb{F}_2) \otimes W$ is the tensor product of two \Cs s each of whose central sequence algebras has a character. We do not know if  the central sequence algebra, $F(A)$, itself has a character. If it does not, then it will serve as a counterexample to Questions \ref{q3}, \ref{q1}, \ref{q2} and \ref{q2'}. 

Let us finally remark that if $A$ and $B$ are unital \Cs s both admitting a character, then $A \otmin B$ has a character. It is not true that $F(A \otmin B)$ has a character if $F(A)$ and $F(B)$ both have a character. Take for example $A = B = W$, where $W$ is as above. Then $F(W)$ has a character, but $W \otimes W$ is purely infinite (by \cite[Theorem 4.1.10 (ii)]{Ror:Encyc}) and is thus simple, separable, unital, nuclear and purely infinite, whence $F(W \otimes W)$ itself is simple and purely infinite, cf.\ \cite{KirPhi:classI}, and therefore characterless.
\end{ex}

%%%%%%%%%%%%%%%%%%%%%%%%%%%%%%%%%%%%%%%%%%%%%%%%%

\section{The splitting property}

\noindent In the previous sections we have discussed when the central sequence algebra $F(A)$ of a (unital) \Cs{} $A$ has a character. The absence of a character can be viewed as a weak divisibility property of $F(A)$ (in fact, the weakest). We shall discuss divisibility properties for \Cs s more formally at the beginning of the next section; and in Section \ref{sec:Z} we shall show that the Jiang-Su algebra embeds into $F(A)$ if (and only if) $F(A)$ has a specific, rather strong, divisibility property. Whereas the various divisibility properties under consideration really are different, they may agree for \Cs s of the form $\otmaxinfty D$, where $D$ is any unital \Cs, or for the central sequence algebra $F(A)$ of any (unital) \Cs{} $A$. 

In this section we investigate a divisibility property which is (slightly) stronger than absence of characters. Recall that an element in a \Cs{} is said to be \emph{full} if it is not contained in any proper closed two-sided ideal. 

\begin{definition}
A \Cs{} $A$ is said to have the \emph{$2$-splitting property} if there exist positive full elements $a,b \in A$ such that $ab = 0$.  
\end{definition}

\noindent
We shall also need the following:  
\begin{definition} \label{def:purely-full}
An element $a$ in a \Cs{} $A$ is said to be \emph{purely full} if $a$ is a positive contraction such that $(a-\ep)_+$ is full for all $\ep \in [0,1)$. 
\end{definition}

\noindent Every simple unital \Cs{} other than $\C$ has the $2$-splitting property. Every non-zero element in a simple \Cs{} is full, and every positive element of norm 1 in a simple \Cs{} is purely full.

For each $\ep >0$ consider the two continuous functions $f_\ep, g_\ep \colon \R^+ \to [0,1]$ given by
\begin{equation} \label{eq:fg}
f_\ep(t) = \begin{cases} \ep^{-1} t, & 0 \le t \le \ep \\ 1, & t \ge \ep \end{cases}, \qquad g_\ep(t) = 1-f_\ep(t), \qquad t \in \R^+.
\end{equation}
If $a$ is a positive full element in a unital \Cs{} $A$, then $(a-\ep)_+$ is full for some $\ep >0$, and the element $f_\ep(a)$ is purely full for any such $\ep$. The hereditary subalgebra $\overline{aAa}$ of any full positive element $a$ therefore contains a purely full positive element. This proves the following:

\begin{lemma} \label{lm:splitting-cf}
A unital \Cs{} has the $2$-splitting property if and only if it  contains two purely full pairwise orthogonal elements. 
\end{lemma}

\noindent
Let us describe some properties of purely full elements.

\begin{lemma} \label{lm:comp-full}
Let $a$ be a positive contraction in a unital \Cs{} $A$. The following conditions are equivalent:
\begin{enumerate}
\item $a$ is purely full. \vspace{.1cm}
\item Each $b \in A$ with $\|a-b\|< 1$ is full.  \vspace{.1cm}
\item $\|\pi(a)\|=1$ for each non-zero \sh{} $\pi$ from $A$ into another \Cs{} $B$.
\end{enumerate}
\end{lemma}

\begin{proof} (i) $\Rightarrow$ (iii). Let $\pi \colon A \to B$ be a non-zero \sh, and put $\ep= \|\pi(a)\|$. Then $0 = (\pi(a)-\ep)_+ = \pi\big((a-\ep)_+\big)$. Hence $(a-\ep)_+$ is not full. Thus $\ep=1$.

(iii) $\Rightarrow$ (ii). Let $b \in A$ with $\|a-b\|<1$ be given, let $I$ be the closed two-sided ideal in $A$ generated by $b$, and let $\pi \colon A \to A/I$ denote the quotient mapping. Then $\|\pi(a)\| = \|a+I\| \le \|a-b\| < 1$. Thus $\pi$ must be the zero mapping, so $I=A$, whence $b$ is full in $A$.

(ii) $\Rightarrow$ (i). This follows from the fact that $\|a-(a-\ep)_+\| \le \ep$ for all $\ep \ge 0$.
\end{proof}

\noindent The negation of having the $2$-splitting property can be reformulated in several ways:

\begin{prop} \label{prop:2-split}
The following conditions are equivalent for any unital \Cs{} $A$.
\begin{enumerate}
\item $A$ does not have the $2$-splitting property.  \vspace{.1cm}
\item Each positive full element in $A$  is invertible in some non-zero quotient of $A$. \vspace{.1cm}
\item For each purely full element $a$ in $A$ there exist a (non-zero) unital \Cs{} $B$ and a unital \sh{} $\pi \colon A \to B$ such that $\pi(a)=1_B$. \vspace{.1cm}
\item For each pair of purely full elements $a,b \in A$ there is a state $\rho$ on $A$ such that $\rho(a) = \rho(b) = 1$.  \vspace{.1cm}
\item $\|ab\|=1$ for each pair of purely full elements $a,b \in A$.
\end{enumerate}
\end{prop}

\begin{proof} (i) $\Rightarrow$ (iii). Let $a$ be a purely full element in $A$, and let $\ep \in [0,1)$ be given. Then $(a-\ep)_+$ is positive and full in $A$. Let $g_\ep \colon \R^+ \to [0,1]$ be as defined in \eqref{eq:fg} above, and let $I_\ep$ be the closed two-sided ideal in $A$ generated by $g_\ep(a)$. As $(a-\ep)_+ \perp g_\ep(a)$ and as (i) holds, we conclude that $I_\ep \ne A$.  Put
$$I = \overline{\bigcup_{\ep < 1} I_\ep}.$$
Then $I$ is a proper ideal in $A$ (because $A$ is unital). Now, $0 = g_\ep(a)+I = g_\ep(a+I)$ in $A/I$ for each $\ep \in [0,1)$, which implies that $a+I$ is the unit of $A/I$.

(iii) $\Rightarrow$ (iv). Let  $\pi \colon A \to B$ be as in (iii) with respect to $a$. Then
$\|\pi(b)\|=1$ by Lemma~\ref{lm:comp-full} (iii). Let $\sigma$ be a state on $B$ such that $\sigma(\pi(b)) = 1$, and let $\rho = \sigma \circ \pi$. Then $\rho$ is a state on $A$, and $\rho(a) = \rho(b) = 1$. 

(iv) $\Rightarrow$ (v). Let $\rho$ be as in (iv). Then $a$ and $b$ are in the multiplicative domain of $\rho$, so $\rho(ab) = \rho(a)\rho(b) = 1$. 

(v) $\Rightarrow$ (i) is trivial, cf.\ Lemma~\ref{lm:splitting-cf}.

(iii) $\Rightarrow$ (ii). Let $a$ be a positive full element in $A$, and let $\ep > 0$ be such that $(a-\ep)_+$  is full. Arguing as below Definition~\ref{def:purely-full} we find that $f_\ep(a)$ is purely full, so $f_\ep(\pi(a)) = \pi(f_\ep(a)) = 1_B$ for some unital \sh{} $\pi \colon A \to B$. This entails that $\pi(a)$ is invertible in $B$.

(ii) $\Rightarrow$ (i). Let $a$ and $b$ be positive full elements in $A$; and let $\pi \colon A \to B$ be a unital \sh{} onto a (non-zero) unital \Cs{} $B$ such that $\pi(a)$ is invertible. Then $\pi(b)$ is non-zero because $b$ is full, so $0 \ne \pi(a) \pi(b) = \pi(ab)$, which shows that $ab \ne 0$. 
\end{proof}

\begin{prop} \label{prop:lift-splitting}
Let $A$ be a separable, simple, unital \Cs{} with $T(A) \ne \emptyset$ and with property (SI). Let $\pi$ denote the quotient mapping $F(A) \to F(A)/(F(A) \cap J(A))$.
\begin{enumerate}
\item An element $a \in A$ is full in $F(A)$ if $\pi(a)$ is full in $F(A)/(F(A) \cap J(A))$. \vspace{.1cm}
\item $F(A)$ has the $2$-splitting property if and only if $F(A)/(F(A) \cap J(A))$ has the $2$-splitting property.
\end{enumerate}
\end{prop}

\begin{proof} (i). Let $a \in A$, suppose that $\pi(a)$ is full in $F(A)/(F(A) \cap J(A))$, and let $I$ be the closed two-sided ideal in $F(A)$ generated by $a$. Upon replacing $a$ by $a^*a \in I$ we may assume that $a$ is positive. To show that $a$ is full in $F(A)$ it suffices to show that $F(A) \cap J(A) \subseteq I$. 

Find elements $x_1, \dots, x_n \in F(A)/(F(A) \cap J(A))$ witnessing that $\pi(a)$ is full, i.e., $1 = \sum_{j=1}^n x_j^*\pi(a)x_j$, and put $\ep = \frac12 \, \big(\sum_{j=1}^n \|x_j\|^2\big)^{-1}$.

Let $f \colon [0,1] \to [0,1]$ be a continuous function such that $f(0) = 0$ and $f(t) = 1$ for $t \ge \ep$, and set $b = f(a)$. Then $b \in I$. Moroever, $\pi(b)^m \ge (\pi(a)-\ep)_+$ for all $m \ge 1$, so
$$\sum_{j=1}^n x_j^* \pi(b)^m x_j \; \ge \; \sum_{j=1}^n x_j^*\big(\pi(a)-\ep\big)_+ x_j \; \ge \; \frac12 \cdot 1$$
for all $m \ge 1$. This entails that $\tau(\pi(b)^m) \ge \ep$ for all $m \ge 1$ and for all tracial states $\tau$ on $F(A)/(F(A) \cap J(A))$. In particular, $\tau(b^m) \ge \ep$ for all $m \ge 1$ and for all tracial states $\tau$ on $F(A)$. Since $A$ has property (SI) we can use Lemma~\ref{prop:SI} to conclude that $e \precsim b$ for all
positive contractions $e \in F(A) \cap J(A)$. This of course shows that $F(A) \cap J(A) \subseteq I$. 

(ii). Suppose that $F(A)/(F(A) \cap J(A))$ has the $2$-splitting property and let $b_1,b_2$ be two full positive elements of $F(A)/(F(A) \cap J(A))$ such that $b_1b_2=0$. Lift $b_1,b_2$ to positive elements $a_1,a_2 \in F(A)$ such that $a_1a_2=0$. Then $a_1,a_2$ are automatically full in $F(A)$ by (i), so $F(A)$ has the $2$-splitting property.
\end{proof}

\begin{ex} \label{ex:splitting}
Suppose that $A$ is a unital \Cs, which is a continuous field over a Hausdorff space $X$ with fibers, $A_x$, isomorphic to $M_n$ for some fixed $n \ge 2$ for all $x \in X$. Then $A$ is of the form $p(C(X) \otimes \cK)p$ for some projection $p \in C(X) \otimes \cK$ of dimension $n$. The primitive ideal space of $A$ is equal to $X$. As no point in $X$ can be the kernel of a character, we see that $A$ has no characters. Denote by $\pi_x \colon A \to A_x \cong M_n$ the fibre map over the point $x \in X$.

An element $a \in A$ is full in $A$ if and only if $\pi_x(a) \ne 0$ for all $x \in X$. A positive element $a$ is purely full in $A$ if and only if $\|\pi_x(a)\| = 1$ for all $x \in X$, cf.\ Lemma~\ref{lm:comp-full}. Hence, by Proposition~\ref{prop:2-split}, $A$ fails to have the $2$-splitting property if and only if whenever $a \in A$ is a positive such that $\|\pi_x(a)\| = 1$ for all $x \in X$, then there exists $x \in X$ such that $\pi_x(a) = 1$. 

As in \cite[Remark 5.8]{RobRor:divisible} consider the case where $X = S^4$ and where $p \in C(S^4) \otimes \cK$ is a $2$-dimensional projection with non-trivial Euler class. Then $A = p(C(S^4) \otimes \cK)p$ does not have the $2$-splitting property. It follows from the main theorem of \cite{DadToms:Z} that the Jiang-Su algebra embeds into $\bigotimes^\infty A$. This in particular implies that some finite tensor power $A \otimes A \otimes \cdots \otimes A$ has the $2$-splitting property. 

We can see this more directly as follows. The projection $p \otimes p \otimes p \in A \otimes A \otimes A$ has a non-zero trivial subprojection $e$ by \cite[9.1.2]{Hus:fibre}, because $$\dim(p \otimes p \otimes p) = 8 > \lceil(\dim((S^4)^3)-1)/2\rceil.$$ Hence $e$ and $p \otimes p \otimes p - e$ are pairwise orthogonal full elements in $A \otimes A \otimes A$.
\end{ex}

%%%%%%%%%%%%%%%%%%%%%%%%%%%%%%%%%%%%%%%%%%%%%%%%%%%%%

\section{Divisibility and comparability properties}

\noindent We have discussed properties of a unital \Cs{} $A$ for which the central sequence algebra $F(A)$ does not have a character, and we raised the question if this condition will imply that the Jiang-Su algebra embeds unitally into $F(A)$, so that $A$ absorbs the Jiang-Su algebra if $A$, in addition, is separable. Absence of characters of a unital \Cs{} is a  \emph{weak divisibility} property of a \Cs. This property was considered in \cite{RobRor:divisible} and shown, in the language of \cite{RobRor:divisible}, to be equivalent to the condition $\wDiv_2(A) < \infty$. We remind the reader that $\wDiv_m(A) \le n$ if and only if there exist $x_1,x_2, \dots, x_n \in \Cu(A)$ such that
$$mx_j \le \langle 1_A \rangle \le x_1+x_2+ \cdots + x_n$$
holds for $j=1, 2, \dots, n$.

 In the previous section we discussed the stronger divisibility property, the so-called $2$-splitting property. It was also considered in \cite{RobRor:divisible}, and shown to hold for a unital \Cs{} $A$ if and only if $\Dec_2(A) < \infty$.  By definition, $\Dec_m(A) \le n$ if and only if there exist elements $x_1,x_2, \dots, x_
m \in \Cu(A)$ such that
$$ x_1+x_2+ \cdots + x_m \le \langle 1_A \rangle \le nx_j$$
holds for all $j=1,3, \dots , m$. 

A still stronger divisibility property for a unital \Cs{} $A$, called the \emph{Global Glimm Halving property}, is the existence of a \sh{} $CM_2 \to A$ whose image is full in $A$. (This, again, is equivalent to the existence of two pairwise orthogonal and full positive elements $a$ and $b$ in $A$ which are \emph{equivalent} in the sense that $a = x^*x$ and $b = xx^*$ for some $x \in A$.) It was shown in \cite{RobRor:divisible}  that $A$ has the Global Glimm Halving property if and only if $\Div_2(A) < \infty$. By definition, $\Div_m(A) \le n$  if and only if there exists an element $x \in \Cu(A)$ such that $mx \le \langle 1_A \rangle \le nx$.

In general one has
$$\Div_m(A) < \infty \implies \Dec_m(A) < \infty \implies \wDiv_m(A) < \infty$$
for all $m$, and $\Div_m(A) < \infty$ implies $\Div_k(A) < \infty$ when $m \ge k$ (and likewise for "$\Dec$" and "$\wDiv$"). The reverse implication do not hold in general, see \cite{RobRor:divisible}. For the \Cs s of interest in this paper we can say more:

\begin{prop} \label{prop:Div-equiv} Let $A$ be a unital \Cs. Then the following are equivalent:
\begin{enumerate}
\item $F(A)$ has the $2$-splitting property. \vspace{.1cm}
\item $\Dec_2(F(A)) < \infty$.  \vspace{.1cm}
\item $\Dec_m(F(A)) < \infty$ for all $m \ge 2$.  \vspace{.1cm}
\item $\Div_m(F(A)) < \infty$ for all $m \ge 2$.  \vspace{.1cm}
\item For each $m \ge 2$ there is a \sh{} $CM_m \to F(A)$ with full image.
\end{enumerate}
\end{prop}

\begin{proof}
The implications (i) $\Leftrightarrow$ (ii) $\Leftarrow$ (iii) $\Leftarrow$ (iv) $\Leftrightarrow$ (v) have been proved in \cite{RobRor:divisible} (and hold for all unital \Cs s in the place of $F(A)$).  

Let us prove (i) $\Rightarrow$ (v). (All tensor products appearing in this proof are the maximal tensor product.) Assume that (i) holds. Let $a,b$ be two positive full elements in  $F(A)$ with $ab=0$. Let $D$ be the unital separable sub-\Cs{} of $F(A)$ generated by $1, a, b$ and elements $t_1, \dots, t_n$, $u_1, \dots, u_n$, and $v_1, \dots, v_n$ such that 
$$1 = \sum_{i=1}^n u_iat_ibv_i.$$
Then $\{at_1b,at_2b, \dots, at_nb\}$ is a full subset of $D$ (i.e., this subset is not contained in any proper closed two-sided ideal of $D$). 

Choose $k$ such that $2^k \ge n$ and find  pairwise orthogonal positive full elements $c_1,c_2, \dots, c_n$ in $\bigotimes_{j=1}^k D$. (Take each $c_i$ of the form $e_1 \otimes \cdots \otimes e_k$ with $e_j = a$ or $e_j = b$.)  Put
$$ x = \sum_{i=1}^n at_ib \otimes c_i \;\in \;\textstyle{\bigotimes_{j=1}^{k+1}}\,  D.$$
Then $x^*x \perp xx^*$ because $a \perp b$. Moreover, $x$ is full in $\bigotimes_{j=1}^{k+1} D$. To see this, note that any closed two-sided ideal $I$ in $\bigotimes_{j=1}^{k+1} D$ which contains $x$ will also contain $x(1\otimes c_i) = at_ib \otimes c_i^2$ for each $i$. As $c_i$, and hence also $c_i^2$, is full, it follows that $I$ contains $at_ib \otimes 1$ for each $i$. Therefore $I$ must be equal to  $\bigotimes_{j=1}^{k+1} D$.

Since $x^*x \perp xx^*$ there is a \sh{} $CM_2 \to \bigotimes_{j=1}^{k+1} D$ which maps $\iota \otimes e_{11}$ to $x^*x$ and $\iota \otimes e_{22}$ to $xx^*$, where $\iota \in C_0((0,1])$ is given by $\iota(t)=t$. As $x$ is full, the image of this \sh{} is full in $\bigotimes_{j=1}^{k+1} D$.

For each $k \ge 1$ there is a full \sh{} $CM_{2^k} \to \bigotimes_{j=1}^k CM_2$, and if $m \le 2^k$, then there is a full \sh{}  $CM_m \to CM_{2^k}$. In summary, for each $m \ge 1$ there exists a full \sh{} $CM_m \to \bigotimes_{j=1}^k D$  for some large enough $k$. 

Finally, it follows from Corollary~\ref{cor:Abel} that there is a unital, and hence full, \sh{} $\bigotimes_{j=1}^k D
\to F(A)$. Hence there is a \sh{} $CM_m \to  F(A)$ whose image is full in $F(A)$.
\end{proof}

\noindent The proposition above also holds with $F(A)$ replaced with $\otmaxinfty D$ where $D$ is any unital \Cs. (The proof is the same.) 

\begin{ques} \label{q4}
Let $D$ be a unital \Cs{} that has no characters. Does it follow that $\otmaxinfty D$ has the $2$-splitting property?
\end{ques}

\noindent If Question \ref{q4} has an affirmative answer, then, by  Corollary~\ref{cor:Abel} and Lemma~\ref{lm:sep-characters}, we could conclude that $F(A)$ has the $2$-splitting property, and hence will satisfy the equivalent conditions of Proposition \ref{prop:Div-equiv}, if and only if $F(A)$ has no characters.  (See also Example~\ref{ex:splitting}.)

We proceed to describe connections between divisibility properties of $F(A)$ and comparability properties of $A$ (and of $F(A)$). 

Let $D$ be a unital \Cs, and let $n$ and $m$ be positive integers. We say that $D$ has the \emph{$(m,n)$-comparison} property if for all $x,y \in \Cu(D)$ with $nx \le my$ one has $x \le y$. Note that $D$ is  almost unperforated (or has strict comparison) if and only if $D$
has $(m,n)$-comparison for all $n>m$. More generally, if $\alpha \ge 1$, then $D$ has $\alpha$-comparison, in the sense of \cite[Definition 2.1]{KirRor:Central-sequence} (see also Definition~\ref{def:alpha} below), if and only if $D$ has $(m,n)$-comparison  for all  positive integers $m,n$ satisfying $n > \alpha m$. 

As in \cite{RobRor:divisible} we say that $D$ is \emph{$(m,n)$-divisible} if $\Div_m(D) \le n$, i.e., if there exists $x \in \Cu(D)$ such that $mx \le \langle 1_D \rangle \le nx$. 

It follows from Proposition~\ref{prop:Div-equiv} that if $A$ is a unital \Cs{} such that $F(A)$ has the $2$-splitting property, then for each $m \ge 1$ there exists $n \ge 1$ such that $F(A)$ is $(m,n)$-divisible. 

We need the following elementary fact about the Cuntz semigroup of a non-separable \Cs.

\begin{lemma} \label{lm:div-D}
Let $A$ be a (possibly non-separable) unital \Cs. 
\begin{enumerate}
\item Let $x,y \in \Cu(A)$ and $n,m \in \N$ be given such that $nx \le my$. It follows that there is a separable unital sub-\Cs{} $D$ of $A$ such that $x$ and $y$ belong to the image of the induced map $\Cu(D) \to \Cu(A)$, and such that $nx \le my$ in $\Cu(D)$. \vspace{.1cm}
\item Let $m,n$ be positive integers and suppose that $A$ is $(m,n)$-divisible. Then there is a separable unital sub-\Cs{} $D$ of $A$ which is $(m,n)$-divisible.
\end{enumerate} 
\end{lemma}

\begin{proof} (i). Let $a,b$ be positive elements in $A \otimes \cK$ which represent $x$ and $y$, respectively. For each $c \in A \otimes \cK$ and each integer $N \ge 1$, let $c \otimes 1_N$  denote the $N$-fold direct sum $c \oplus c \oplus \cdots \oplus c$, and identify this with an element of $A \otimes \cK$ (by choosing an isomorphism $\cK \cong M_N \otimes \cK$). 

 For each $k \ge 1$ there exists an element $d_k \in A \otimes \cK$ such that 
$$d_k^*\big( b \otimes 1_m\big)d_k = (a-1/k)_+ \otimes 1_n.$$
Let $D$ be any unital separable sub-\Cs{} of $A$ such that $D \otimes \cK$ contains the elements $a,b, a \otimes 1_n, b \otimes 1_m$ and $d_k$ for all $k \ge 1$. Then $D$ has the desired properties. 

(ii). If $A$ is $(m,n)$-divisible, then there exist $x,y \in \Cu(A)$ such that $mx \le \langle 1_A \rangle \le ny$. By applying part (i) to both of these inequalities (representing $\langle 1_A \rangle$ with $1_A$) we get the desired separable $(m,n)$-divisible  sub-\Cs{} $D$ of $A$. 
\end{proof}

\noindent Part (ii) of the proposition below, that relates divisibility properties of $F(A)$ to comparability properties of $A$ and of $F(A)$, is essentially contained in \cite[Lemma 6.1]{RobRor:divisible}. For the convenience of the reader we include a proof. 

\begin{prop} \label{prop:div-comp}  Let $A$ be a unital \Cs.
\begin{enumerate}
\item If $F(A)$ has the $2$-splitting property, then for each integer $m \ge 2$ there exists an integer $n \ge 1$ such that $F(A)$ is $(m,n)$-divisible. \vspace{.1cm}
\item Suppose that $m,n$ are  positive integers such that $F(A)$ is  $(m,n)$-divisible. Then $A$ and $F(A)$ have the $(m,n)$-comparison property, and $A$ is $(m,n)$-divisible.
\end{enumerate}
\end{prop}

\begin{proof}  (i). Set $n = \Div_m(F(A))$, which is finite by Proposition~\ref{prop:Div-equiv}.

(ii). Pick a separable unital sub-\Cs{} $D$ of $F(A)$ which is $(m,n)$-divisible, cf.\ Lemma~\ref{lm:div-D}, and let $\omega$ be a free ultrafilter which realizes $F(A) = F_\omega(A) = A_\omega \cap A'$. 

Let us first show that $A$ has the $(m,n)$-comparison property. Let $x,y \in \Cu(A)$ be given such that $nx \le my$. Then $x \otimes \langle 1_D \rangle \le y \otimes \langle 1_D \rangle$ in $\Cu(A \otmax D)$ by \cite[Lemma~6.1~(i)]{RobRor:divisible}. Let $\varphi \colon A \otmax D \to A_\omega$ be the natural \sh, and let $P$ be the image of $\varphi$. Then $A \subseteq P \subseteq A_\omega$ and
$$\Cu(\varphi)(x \otimes \langle 1_D\rangle) = x \in \Cu(P), \qquad \Cu(\varphi)(y \otimes \langle 1_D \rangle) = y \in \Cu(P).$$
It follows that $x \le y$ in $\Cu(P)$. We can now use Lemma~\ref{lm:A-D-comparison} to conclude that $x \le y$ in $\Cu(A)$.

Suppose now that $x,y \in \Cu(F(A))$ are given such that $nx \le my$. Use Lemma~\ref{lm:div-D} (i) to find a separable sub-\Cs{} $B$ of $F(A)$ such that $x$ and $y$ belong to  $\Cu(B)$ and satisfy $nx \le my$ in $\Cu(B)$. It then follows from \cite[Lemma~6.1~(i)]{RobRor:divisible} that $x \otimes \langle 1_D \rangle \le y \otimes \langle 1_D \rangle$ in $\Cu(B \otmax D)$. 

Use Proposition~\ref{prop:Abel} to find a unital \sh{} $\rho \colon B \to F(A) \cap D'$, and then define a unital \sh{} $\varphi \colon B \otmax D \to F(A)$ by $\varphi(b \otimes d) = \rho(b) d$, for $b \in B$ and $d \in D$. Then 
$$\Cu(\varphi)(x \otimes \langle 1_D \rangle) = x, \qquad \Cu(\varphi)(y \otimes \langle 1_D \rangle) = y.$$
Hence $x \le y$ in $\Cu(F(A))$. 

Finally, since $F(A)$ is $(m,n)$-divisible, so is $A_\omega$, i.e., $\Div_m(A_\omega) \le n$. It then follows from  \cite[Proposition 8.4 (i)]{RobRor:divisible} that $\Div_m(A) \le n$, whence $A$ is $(m,n)$-divisible.
\end{proof}

\noindent Recall that an element $x$ in an ordered additive semigroup $S$ is \emph{properly infinite} if $2x \le x$. If this hold, then $kx \le x$ for all integers $k \ge 1$. 

\begin{lemma}  \label{lm:stable-pi}
Let $A$ be a \Cs{} which has $(m,n)$-comparison for some positive integers $m,n$ with $m \ge 2$. For each $x \in \Cu(A)$ and for each integer $k \ge 2$, if $kx$ is properly infinite, then $x$ is properly infinite.
\end{lemma}

\begin{proof} Let $\ell \ge 1$ be an integer. Observe that $\ell x$ is properly infinite if and only if $\ell'x \le \ell x$ for all integers $\ell' \ge 1$. Let $\ell \ge 1$ be the least integer such that $\ell x$ is properly infinite. Assume, to reach a contradiction, that $\ell \ge 2$. Put $y= (\ell-1)x$ and $z = \ell x$. Then $nz = n\ell x \le \ell x \le m(\ell-1)x = my$. It follows that $z \le y$, i.e., that $\ell x \le (\ell-1)x$. But then $\ell'x \le \ell x \le (\ell-1)x$ for all $\ell' \ge 1$, which shows that $(\ell-1)x$ is properly infinite, a contradiction.
\end{proof}

\begin{prop} \label{prop:sf-pi}
Let $A$ be a unital \Cs{} and suppose that $F(A)$ has the $2$-splitting property. Then the following holds:
\begin{enumerate}
\item If $x \in \Cu(A)$ is such that $kx$ is properly infinite for some integer $k \ge 1$, then $x$ is properly infinite. \vspace{.1cm}
\item If $p$ is a projection in $A \otimes \cK$ and if some multiple $p \oplus p \oplus \cdots \oplus p$ of $p$ is properly infinite, then $p$ is properly infinite. \vspace{.1cm}
\item If $A$ is simple, then either $A$ is stably finite or $A$ is purely infinite.
\end{enumerate}
\end{prop}

\begin{proof} It follows from Propositions~\ref{prop:div-comp} that $A$ has the $(2,n)$-comparison property for some integer $n \ge 1$. Hence (i) follows from Lemma~\ref{lm:stable-pi}. Part (ii) follows from part (i) because a projection $q \in A \otimes \cK$ is properly infinite if and only if $\langle q \rangle$ is properly infinite in $\Cu(A)$. 

(iii). Suppose that $A$ is simple and not stably finite. We must show that $A$ is purely infinite. It suffices to show that each (non-zero) positive element $a \in A$ is properly infinite, or, equivalently, that $x$ is properly infinite in $\Cu(A)$ for all $x \in \Cu(A)$. By (i) it suffices to show that $Nx$ is properly infinite for some $N$. 

By the assumption that $A$ is not stably finite there exists $n \ge 1$ such that $M_n(A)$ is infinite, and hence properly infinite (because $A$ is assumed to be simple). Hence $n \langle 1_A \rangle$ is properly infinite in $\Cu(A)$. This entails that $y \le n\langle 1_A \rangle$ for all $y \in \Cu(A)$. Moreover, by simplicity of $A$, we know that $kx \ge \langle 1_A \rangle$ for some integer $k \ge 1$. Put $N=nk$. Then $k'x \le n \langle 1_A \rangle \le Nx$ for every integer $k' \ge 1$. This shows that $Nx$ is properly infinite.
\end{proof}

\noindent The conclusions of Proposition~\ref{prop:sf-pi} hold for any \Cs{} of real rank zero with the strong Corona Factorization Property, cf.\ \cite[Theorem 5.14 and Corollary 5.16]{OPR:CFP-RR0}. Thus, when specializing to \Cs s of real rank zero, we can in Proposition~\ref{prop:sf-pi} relax the assumption that  $F(A)$ has the $2$-splitting property to the (formally) weaker assumption that  $F(A)$ has no characters, cf.\ Theorem~\ref{thm:CFP}. 

%%%%%%%%%%%%%%%%%%%%%%%%%%%%%%%%%%%%%%%%%%%%%%%%%%%%%%%%%%%

\section{Embedding the Jiang-Su algebra} \label{sec:Z}

\noindent We proved in the previous section that if $A$ is a unital \Cs{} such that $F(A)$ has the $2$-splitting property, i.e., it contains two full pairwise orthogonal elements, then, for each integer $m \ge 2$, there is a full \sh{} $CM_m \to F(A)$. Moreover, for each $m \ge 2$ there exists an integer $n \ge 1$ such that $F(A)$ is $(m,n)$-divisible, i.e., $mx \le \langle 1 \rangle \le n x$  for some $x$ in $\Cu(F(A))$. 

We give here a sufficient (and also necessary) divisibility condition on $F(A)$ that will ensure the existence of a unital embedding of the Jiang-Su algebra into $F(A)$, and hence imply $\cZ$-stability of $A$. We emphasize that this condition, at least formally, is stronger than the splitting property considered in the previous sections, which again, formally, is stronger than the condition that $F(A)$ has no characters. 

\begin{definition}[{\mbox{cf.}} {\cite[Definition 2.1]{KirRor:Central-sequence}}] \label{def:alpha}
Let $D$ be a unital \Cs. 
\begin{enumerate}
\item We say that $D$ has the \emph{$\alpha$-comparison property} if the following holds. 
$$ \forall x,y \in \Cu(D) \; \forall n,m \in \N: nx \le my \; \; \text{and} \; \; n > \alpha m \implies x \le y.$$
\item We say that $D$ has the \emph{$\alpha$-divisibility property} if  for all $x \in \Cu(D)$ and for all integers $n, m \ge 1$ such that $n > \alpha m$ there exists $y \in \Cu(D)$ such that $my \le x \le ny$. 
\end{enumerate}
\end{definition}

\noindent We also remind the reader of the \emph{asymptotic divisibility constant}  from 
\cite[Section 4]{RobRor:divisible} which for each unital \Cs{} $D$ is defined to be 
$$\Div_*(D) \;  := \liminf_{m \to \infty} \frac{\Div_m(D)}{m}.$$
We know from Proposition~\ref{prop:Div-equiv} that $\Div_m(A) < \infty$ and $\Div_m(F(A)) < \infty$ for all $m \ge 2$ if $F(A)$ has the $2$-splitting property. However, we do not know if this also implies that $\Div_*(F(A)) < \infty$.

\begin{prop} \label{prop:alpha-pure}
 Let $A$ be a separable unital \Cs{} for which $\alpha := \Div_*(F(A)) < \infty$. 
Then $A$ and $F(A)$ are $\alpha$-divisible and have $\alpha$-comparison. 

In particular, if $m = \lceil \alpha \rceil-1$, then $A$ is $(m,m)$-pure (in the sense of Winter, \cite{Winter:Z}). 
\end{prop}

\begin{proof}  Let $n,m$ be positive integers such that $n > \alpha m$. It follows from \cite[Proposition~4.1]{RobRor:divisible} (and its proof) that $F(A)$ is $(m,n)$-divisible. Next, it follows from Proposition~\ref{prop:div-comp} that $A$ and  $F(A)$ have $(m,n)$-comparison and that $A$ is $(m,n)$-divisible. This shows that $A$ and $F(A)$ have the $\alpha$-comparison and $\alpha$-divisibility properties. 
\end{proof}

\noindent If we combine Proposition~\ref{prop:alpha-pure} above with the main theorem from Winter's seminal paper, \cite{Winter:Z}, we obtain:

\begin{prop}  Let $A$ be a separable simple unital \Cs{} with locally finite nuclear dimension. If $\Div_*(F(A)) < \infty$, then $A \cong A \otimes \cZ$.
\end{prop}

\noindent We conclude this paper with a result saying that $\cZ$-stability of an arbitrary unital separable \Cs{} $A$ is equivalent to a (sufficiently strong) divisibility property of $F(A)$.  It is well-known, as remarked in Theorem~\ref{thm:Z-absorption}, that $A$ is $\cZ$-stable if (and only if) there is a unital \sh{} from the dimension drop \Cs{} $I(2,3)$  into $F(A)$. It was shown in \cite[Proposition 5.1]{RorWin:Z} that there is a unital 
\sh{} from $I(2,3)$ into a unital \Cs{} $D$ \emph{with stable rank one} if (and only if) $\Div_2(D) \le 3$, i.e., if there exists $x \in \Cu(D)$ such that $2x \le \langle 1_D \rangle \le 3x$. However, in general, $F(A)$ does not have stable rank one. 

It is also shown in \cite[Proposition 5.1]{RorWin:Z} that there is a unital \sh{} from $I(2,3)$ into a unital \Cs{} $D$ if, for some $\ep>0$,  there exist pairwise orthogonal and equivalent positive contractions $a,b$ in $D$ such that $1-a-b \precsim (a-\ep)_+$. (This does not require that $D$ has stable rank one.) Using this fact we prove:

\begin{lemma} \label{lem:I(2,3)} Let $D$ be a unital \Cs{} and suppose that $D$ is $(2n,N)$-divisible and $D$ has $\alpha$-comparison, where $n$ and $N$ are positive integers and $\alpha$ a real number satisfying
$$2n < N < 3n, \qquad 1 \le \alpha < \frac{n}{N-2n}.$$ 
Then there is a unital \sh{} from $I(2,3)$ into $D$.
\end{lemma}

\begin{proof} Follow the proof of "(i) $\Rightarrow$ (ii)" of \cite[Proposition 5.1]{RorWin:Z} to obtain pairwise equivalent and pairwise orthogonal positive elements $e_1,e_2, \dots, e_{2n}$ in $A$ such that $N\langle e_j \rangle \ge \langle 1_D \rangle$. Choose $\delta >0$ such that $N\langle (e_j-\delta)_+ \rangle \ge \langle 1_D \rangle$, and choose $0 < \ep < \delta$. Let $f_\ep \colon \R^+ \to [0,1]$ be as defined in \eqref{eq:fg}, cf.\  the proof of  \cite[Lemma 4.5]{RorWin:Z}. Put 
$$a_0 = (e_1 -\ep)_+ + (e_2-\ep)_+ +\cdots + (e_{n}-\ep)_+, \qquad b_0 = (e_{n+1} -\ep)_+ + (e_{n+2}-\ep)_+ +\cdots + (e_{2n}-\ep)_+,$$
$$c =  1_D - f_\ep(e_1+e_2+ \cdots + e_{2n}).$$
Then $a_0, b_0$, and $c_0$ are pairwise orthogonal, and
$$ \big(a_0-(\delta-\ep)\big)_+ = (e_1 -\delta)_+ + (e_2-\delta)_+ +\cdots + (e_{n}-\delta)_+.$$
Put $x = \langle (a_0-(\delta-\ep))_+ \rangle$ and $y = \langle c \rangle$ in $\Cu(D)$. If $\rho$ is a state on $\Cu(D)$ normalized at $u = \langle 1_D \rangle$, then 
$$n/N \le \rho(x) \le 1/2, \qquad \rho(2x+y) \le 1.$$
It therefore follows that
$$\alpha \rho(y) \le \alpha (1-2\rho(x)) < \rho(x)$$
for all states $\rho$ on $\Cu(D)$ normalized at $u$, and hence also for all states $\rho$ on $\Cu(D)$ normalized at $x$. (We have here used the relations satisfied by the numbers $n,N$ and $\alpha$.) Since $D$ has $\alpha$-comparison this implies that $y \le x$ in $\Cu(D)$, cf.\ \cite[Lemma 2.3]{KirRor:Central-sequence}. Hence $ c \precsim (a_0-(\delta-\ep))_+$.

Put
$$a= f_\ep(e_1) + f_\ep(e_2) + \cdots + f_\ep(e_n), \qquad b = f_\ep(e_{n+1}) + f_\ep(e_{n+2}) + \cdots + f_\ep(e_{2n}).$$
Then $a \sim b$, $a \perp b$, and $a$ is Cuntz equivalent to $e_1+e_2+ \cdots + e_n$. There is $\eta >0$ such that $(a_0-(\delta-\ep))_+ \precsim (a-\eta)_+$. Thus
$$1-a-b = 1_D - f_\ep(e_1+e_2+ \cdots + e_{2n}) = c \precsim (a_0-(\delta-\ep))_+ \precsim (a-\eta)_+.$$
The existence of a unital \sh{} from $I(2,3)$ into $D$ now follows from the implication "(ii) $\Rightarrow$ (iv)" of \cite[Proposition 5.1]{RorWin:Z}.
\end{proof}

\begin{lemma} \label{lem:I(2,3)-2} Let $D$ be a unital \Cs{} and suppose that $D$ is $\alpha$-divisible and has $\alpha$-comparison for some
$$\alpha < 1 + \sqrt{3}/2.$$
Then there is a unital \sh{} from $I(2,3)$ into $D$.
\end{lemma}

\begin{proof} By the choice of $\alpha$ there exist positive integers $n,N$ such that 
$$2\alpha n < N, \qquad 2n < N, \qquad \alpha < \frac{n}{N-2n}.$$ 
As $D$ is $\alpha$-divisible, the first inequality implies that $D$ is $(2n,N)$-divisible. The claim now follows from Lemma~\ref{lem:I(2,3)}. 
\end{proof}

\noindent We can now express $\cZ$-stability of an arbitrary separable unital \Cs{} in terms of a divisibility property of its central sequence algebra:

\begin{prop} The following three conditions are equivalent for every unital separable \Cs{} $A$:
\begin{enumerate}
\item $A \cong A \otimes \cZ$. \vspace{.1cm}
\item $\Div_*(F(A)) \le 1$. \vspace{.1cm}
\item $\Div_*(F(A)) < 1+\frac{\sqrt{3}}{2}$.
\end{enumerate}
\end{prop}

\begin{proof} (i) $\Rightarrow$ (ii). If (i) holds, then $\cZ$ embeds unitally into $F(A)$, so
$$\Div_*(F(A)) \le \Div_*(\cZ) = 1.$$

(ii) $\Rightarrow$ (iii) is trivial. (iii) $\Rightarrow$ (i). If $\alpha := \Div_*(F(A)) < 1+\frac{\sqrt{3}}{2}$, then $F(A)$ is $\alpha$-divisible and has $\alpha$-comparison by Proposition~\ref{prop:alpha-pure}. Hence there is a unital \sh{} from $I(2,3)$ into $F(A)$ by Lemma~\ref{lem:I(2,3)-2}. This implies that (i) holds, cf.\ Theorem~\ref{thm:Z-absorption}. 
\end{proof}

{\small{
\bibliographystyle{amsplain}
\providecommand{\bysame}{\leavevmode\hbox to3em{\hrulefill}\thinspace}
\providecommand{\MR}{\relax\ifhmode\unskip\space\fi MR }
\providecommand{\MRhref}[2]{%
  \href{http://www.ams.org/mathscinet-getitem?mr=#1}{#2}
}
\providecommand{\href}[2]{#2}

}}

\bigskip \bigskip
\noindent
\address{Institut f{\"u}r Mathematik,
Humboldt Universit{\"a}t zu Berlin, Unter den Linden 6,\\
D--10099 Berlin, Germany}\\
\email{kirchbrg@mathematik.hu-berlin.de}

\bigskip

\noindent
\address{Department of Mathematics, University of Copenhagen\\
Universitetsparken 5, 2100 Copenhagen, Denmark}\\
\email{rordam@math.ku.dk}

\end{document}